\documentclass[leqno,12pt]{amsart}

\usepackage{amsmath,amstext,amssymb,amsopn,amsthm,mathrsfs}
\usepackage{bbm}
\usepackage{verbatim}
\usepackage{enumerate}  
\usepackage{enumitem}
\usepackage{color}

\newcommand*\RR{\mathbb{R}}

\newcommand*\NN{\mathbb{N}}
\newcommand*\ZZ{\mathbb{Z}}

\newcommand*\ven{\vert n\vert}
\newcommand*\al{\alpha}
\newcommand*\be{\beta}
\newcommand*\bal{\bar{\alpha}}
\newcommand*\bbe{\bar{\beta}}
\newcommand*\te{\theta}
\newcommand*\va{\varphi}

\newcommand*\fun{\varphi_n}
\newcommand*\funk{\varphi_k}
\newcommand*\funi{\varphi_{n_i}}
\newcommand*\Lfun{\mathcal{L}_n^\al}
\newcommand*\Lfuni{\mathcal{L}^{\alpha_i}_{n_i}}
\newcommand*\Lfunk{\mathcal{L}^{\alpha}_{k}}

\newcommand*\hfun{\varphi_n^\al}
\newcommand*\hfuni{\varphi^{\alpha_i}_{n_i}}
\newcommand*\hfunk{\varphi^{\alpha}_{k}}
\newcommand*\hen{h_n^\lambda}

\newcommand*\jfun{\phi_n^{\al,\be}}
\newcommand*\jfunk{\phi_k^{\al,\be}}
\newcommand*\jfuni{\phi_{n_i}^{\al_i,\be_i}}

\newcommand*\jpolk{P_k^{\al,\be}}

\newcommand*\jtpolk{\mathcal{P}_k^{\al,\be}}

\newcommand*\Rop{R_r^{\alpha}}
\newcommand*\Ropi{R_r^{\alpha_i}}
\newcommand*\jRop{R_r^{\al,\be}}

\newcommand*\htRop{\tilde{R}_r^{\lambda}}

\newcommand*\ve{\varepsilon}

\title[Hardy's inequality in $H^p$]
{Sharp Hardy's inequality for orthogonal expansions in~$H^p$~spaces}

\author[P{.} Plewa]{Pawe\l{} Plewa}
\address{Pawe\l{} Plewa \newline
			Faculty of Pure and Applied Mathematics,\newline
      Wroc\l{}aw University of Science and Technology       \newline
      Wyb{.} Wyspia\'nskiego 27,
      50--370 Wroc\l{}aw, Poland      
      }
\email{pawel.plewa@pwr.edu.pl}

\allowdisplaybreaks

\theoremstyle{plain}
\newtheorem{thm}{Theorem}[section]

\newtheorem{lm}[thm]{Lemma}
\newtheorem{prop}[thm]{Proposition}

\theoremstyle{definition}

\theoremstyle{remark}
\newtheorem*{rem*}{Remark}
\newtheorem{rem}[thm]{Remark}

\numberwithin{equation}{section}

\theoremstyle{plain}

\usepackage[dvipsnames]{xcolor}

\newcounter{comcount}

\begin{document}
\begin{abstract}
Hardy's inequality on $H^p$ spaces, $p\in(0,1]$, in the context of orthogonal expansions is investigated for general basis on a subset of $\RR^d$ with Lebesgue measure. The obtained result is applied to various Hermite, Laguerre, and Jacobi expansions. For that purpose some delicate estimates of the higher order derivatives for the underlying functions and of the associated heat kernels are proved. Moreover, sharpness of studied Hardy's inequalities is justified by a construction of an explicit counterexample, which is adjusted to all considered settings.
\end{abstract}

\maketitle
\footnotetext{
\emph{2010 Mathematics Subject Classification:} Primary: 42C10; Secondary: 42B30, 42B05, 33C45.\\
\emph{Key words and phrases:} Hardy's inequality, Hardy spaces, Laguerre expansions, Hermite expansions, Jacobi expansions. \\
The paper is a part of author's doctoral thesis written under the supervision of Professor Krzysztof Stempak.
}

\section{Introduction}
The classical Hardy inequality (see \cite{Hardy&Littlewood_1927_MA}) for Fourier coefficients states that
\begin{equation}\label{eq:classical_Hardy_inequ}
\sum_{k\in\ZZ} \frac{|\hat{f}(k)|}{|k|+1}\lesssim \Vert f\Vert_{{\rm Re}\, H^1}, 
\end{equation}
where ${\rm Re}\, H^1$ is the real Hardy space, where belong the real parts of functions in the Hardy space $H^1(\mathbb{D})$. Here $\mathbb{D}$ denotes the unit disk in the plane. Analogues of \eqref{eq:classical_Hardy_inequ} were considered by Kanjin \cite{Kanjin_1997_BLMS}, and $\hat{f}(k)$ were replaced by the expansion coefficients in two orthonormal bases: the Hermite and standard Laguerre function systems. In general, such inequalities are of the form
\begin{equation}\label{eq:13}
\sum_{k\in\NN} \frac{|\langle f,\funk\rangle|}{(k+1)^E}\lesssim \Vert f\Vert_{H^1},
\end{equation}
where $\funk$ is an orthonormal basis, $\langle\cdot,\cdot\rangle$ denotes the inner product in the associated space $L^2$, $H^1$ is an appropriate Hardy space, and $E$ is a positive number which we refer to as the admissible exponent. The difficulty in establishing versions of \eqref{eq:13} is twofold. Firstly, given an orthonormal basis one can ask if such inequality holds for a certain $E$. Secondly, there is a question about sharpness of the admissible exponent. We say that $E$ is sharp if it is the smallest for which \eqref{eq:13} holds. Moreover, some generalization of \eqref{eq:13} are possible, such as replacing $H^1$ by $H^p$, $p\in(0,1]$, or considering the multi-dimensional situation.

In the last two decades many authors were interested in various Hardy's inequalities. As mentioned above, Kanjin initiated the study of version of \eqref{eq:13} for the Hermite functions (he obtained $E=29/36$) and the standard Laguerre functions ($E=1$). For the latter system Satake \cite{Satake_2000_JMSJ} generalized this result for $p\in(0,1)$ with $E=2-p$, and for the former expansions Radha \cite{Radha_2000_TJM} extended investigated the multi-dimensional situation $d\geq 1$ with $E>(17d+12)/(24+12d)$. Few years later Radha and Thangavelu \cite{Radha_Thangavelu_2004_PAMS} proved Hardy's inequality associated with Hermite expansions for $d\geq 2$ and $p\in(0,1]$ with the admissible exponent $E=3d(2-p)/4$. The lacking case $d=1$ was partially covered by Balasubramanian and Radha \cite{Balasubramanian_Radha_2005_JIPAA}, but the exponent was strictly larger than the expected value $3(2-p)/4$ (see also Kanjin \cite{Kanjin_2011_JMSJ}). The inequality with this admissible exponent was proved ten years later by Z. Li, Y. Yu, and Y. Shi \cite{LiYuShi_2015_JFAA}. Moreover, the Jacobi trigonometric function expansions were studied by Kanjin and K. Sato \cite{Kanjin_Sato_2001_BLMS(Jacobi_Paley),Kanjin_Sato_2004_MIA(Jacobi_Hardy)}. There are also some other paper concerning various Hardy's inequalities in the context of orthogonal expansions, see for instance \cite{Ciaurri&Roncal&Thangavelu_2018_PEMS, LiShi_2014_CA,Shi&Li_2016_JMSJ_2, Shi&Li_2016_JMSJ}.

The author have already written a few articles in this topic. In \cite{Plewa_2019_JFAA} the system of Laguerre functions of Hermite type was studied. Secondly, in \cite{Plewa_2018_sharpLaguerre_arxiv} a general multi-dimensional method of proving Hardy's inequalities was introduced. It consists in estimating kernels of certain family of operators closely related to the associated heat semigroup. The method was applied to two Laguerre systems: standard and of convolution type. We stress that in the latter the underlying measure is not Lebesgue measure. Furthermore, in the same paper sharpness of obtained admissible exponents was proved. Up to our knowledge, it was the first explicit construction of such counterexamples known in the associated literature. On the other hand, the long study of Hardy's inequality for Hermite expansions were concluded by the author \cite{Plewa_2020_TJM}, where he justifies that the know exponent $E=3d/4$ (for $p=1$) was sharp. Finally, four Jacobi systems were also investigated, see \cite{Plewa_2019_JAT}.

In this paper we prove Hardy's inequalities in the framework of various orthogonal function systems including generalized Hermite, standard Laguerre, Laguerre of Hermite type, and trigonometric Jacobi expansions in $H^p$ spaces, $p\in(0,1]$. We focus on systems associated with Lebesgue measure. The main reason behind this restriction is that the atomic $H^p$ spaces are not well defined for all $p\in(0,1)$ when the underlying measure is more arbitrary and only assumed to be doubling. On the other hand, if $p=1$, then there is no need for such restraint, see \cite[Theorem~2.2]{Plewa_2018_sharpLaguerre_arxiv}.

Although we prove Hardy's inequality for certain orthogonal systems, we are interested in establishing a general method which works in the known settings. Therefore, we enhance the approach from \cite{Plewa_2018_sharpLaguerre_arxiv} and adjust it for the case $p\in(0,1]$. It requires estimating derivatives of an arbitrary order of the kernels $R_r(x,y)$ (see \eqref{R_def}). In most cases it turns out to be not so difficult as one could expect once we have analogous asymptotics for the functions composing the considered basis. However, for the Laguerre expansions of Hermite type it is much more involved, see the proof of Proposition \ref{prop:Laguerre_Hermite_R_deriv}. This result can be viewed in terms of the heat kernel, see Subsection \ref{subS:heat_kern_estim}. Moreover, by some minor modifications we were able to add the parameter $s\in[p,2]$ in Theorem \ref{thm:Hp_general}. 

Another novelty of the paper is the unified approach to sharpness. Instead of constructing separate counterexamples in each setting, we construct one sequence of piecewise constant atoms which, with an additional assumption, justifies that the admissible exponent is sharp. In order to verify the added condition in the specific settings we have to subtly estimate the derivatives of the functions in orthonormal basis, see Lemmas \ref{lm:standard_Laguerre_sharp_estimates}, \ref{lm:Laguerre_Hermite_sharp_estimates}, and \ref{lm:Jacobi_fun_sharp_estimates}. These results can be interesting on their own.

The main result of the paper is Hardy's inequality for general setting, see Theorem \ref{thm:Hp_general}, and sharpness of the admissible exponent, see Propositions \ref{prop:general_sharpness} and \ref{prop:general_sharpness_2}. This theorem is then applied in several settings, see Theorems \ref{thm:main_Laguerre_standard}, \ref{thm:main_Laguerre_Hermite}, \ref{thm:main_generalized_Hermite}, and \ref{thm:main_Jacobi_fun}, which generalize many of already known in the literature results (see \cite{Balasubramanian_Radha_2005_JIPAA, Kanjin_1997_BLMS,Kanjin_2011_JMSJ, Kanjin_Sato_2004_MIA(Jacobi_Hardy),LiYuShi_2015_JFAA,Radha_2000_TJM, Radha_Thangavelu_2004_PAMS,Satake_2000_JMSJ}), but also answer some open questions (for instance sharpness or multi-dimensional inequality on $H^p$ for Laguerre expansions).    

Organization of the paper is as follows. In Section \ref{S:Hardy's_Ineq} we prepare the necessary tools to prove Hardy's inequality, like Hardy, BMO, and Lipschitz spaces. Moreover, we enhance the method from \cite{Plewa_2018_sharpLaguerre_arxiv} so it works for $H^p$ spaces with $p\in(0,1]$. Furthermore, we construct a counterexample to justify that the obtained formula for the admissible exponent is sharp. In Section \ref{S:Laguerre_standard} we discuss the standard Laguerre functions and estimate their derivatives near zero. This allows us to apply the general theorem. Section \ref{S:Laguerre_Hermite} is devoted to Laguerre expansions of Hermite type. Similarly as before we estimate the derivatives of the functions from the basis. However, this time it is not immediate to obtain such bounds for the corresponding kernels $R_r(x,y)$. For that purpose we need to use the integral formula for the Bessel function, see \eqref{eq:36} and Proposition \ref{prop:Laguerre_Hermite_R_deriv}. We also interpret this estimate in term of the heat kernel. Moreover, we deduce Hardy's inequality for the generalize Hermite framework. Lastly, in Section \ref{S:Jacobi_trigonom_fun} is analysed the Jacobi trigonometric function system. 

\subsection*{Notation}
Throughout this paper $d\geq 1$ denotes the dimension, $u,v$ are real variables, and $x=(x_1,\ldots,x_d)$, $y=(y_1,\ldots,y_d)$ are vectors from $\RR^d$ or $\RR^d_+=(0,\infty)^d$. We use $k,i,j$ for integers belonging to $\NN=\NN_+\cup\{0\}=\{0,1,\ldots\}$, and $n=(n_1,\ldots,n_d)\in\NN^d$ for the multi-indices. Let $\ven=n_1+\ldots+n_d$ stand for the length of $n$. We denote the type indices $\al$ and $\be$ with the same symbol in both situation $d=1$ and $d\geq 1$. In the latter case we use the same convention as for $n$. For any $u\in\RR$ we denote the largest integer not greater than $u$ by $\lfloor u\rfloor$, and the smallest integers not smaller than $u$ by $\lceil u\rceil$. We write $\lesssim$ for inequalities with non-negative entries which hold with a multiplicative constant. It may depend on the quantities stated beforehand, but not on the ones quantified afterwards. If $X\lesssim$ and $Y\lesssim X$ simultaneously, then we write $X\simeq Y$. 

\subsection*{Acknowledgements}
Research supported by the National Science Center of Poland, NCN grant no. 2018/29/N/ST1/02424.\\
The author is deeply grateful to Professor Krzysztof Stempak for his constant help during the preparation of this paper, and for his numerous comments and suggestions.

\section{Hardy's inequality}\label{S:Hardy's_Ineq}

In this section we develop a method of proving Hardy's inequality on $H^p$ spaces, $0<p\leq 1$, associated with orthonormal expansions. This is a generalization of the idea described in \cite[Section ~2]{Plewa_2018_sharpLaguerre_arxiv}. However, Hardy spaces, even in the sense of Coifman-Weiss \cite{Coifman_Weiss_1977_BAMS}, are not well defined for all $p$ if the underlying measure is only assumed to be doubling. Hence, we will focus our attention only on orthogonal expansions in $L^2(X)$, where $X$ is a subset of $\RR^d$ equipped with Lebesgue measure.

\subsection{Hardy spaces}

Recall that given any Schwartz function $\Phi$ such that $\int \Phi\neq 0$, one can define the Hardy space $H^p(\RR^d)$, $p<0\leq 1$, as the space of all distributions satisfying
\begin{equation*}
	\sup_{t>0} |f\ast \Phi_t|\in L^p(\RR^d),
\end{equation*}
where $\Phi_t(x)=t^{-d}\Phi(x/t)$. The $p$-th norm of the quantity above can be taken as a "norm" in $H^p(\RR^d)$. We remark that $\Vert\cdot\Vert_{H^p(\RR^d)}$ is indeed a norm only for $p=1$. In fact, $H^1(\RR^d)$ is a Banach space. In general, if $p\leq 1$, then $\Vert \cdot\Vert^p_{H^p(\RR^d)}$ is subadditive, hence $d(f,g)=\Vert f-g\Vert^p_{H^p(\RR^d)}$ defines a complete metric on $H^p(\RR^d)$.

A measurable function $a$ supported in a Euclidean ball $B$ is called a $(p,q)$-atom for $0<p\leq 1$ and $q\in[1,\infty]$, $p<q$, if it satisfies
\begin{equation*}
\int_B a(x)x^{n}\,dx=0\qquad \text{and}\qquad\Vert a\Vert_{L^q(\RR^d)}\leq |B|^{\frac{1}{q}-\frac{1}{p}},
\end{equation*}
where $x^n=x_1^{n_1}\ldots x_d^{n_d}$, $\ven\leq \lfloor d(p^{-1}-1)\rfloor $, and $|B|$ denotes the Lebesgue measure of $B$.
Every $f\in H^p(\RR^d)$ admits an atomic decomposition, namely there exist a sequence of $(p,q)$-atoms $\{a_j\}_{j\in\NN}$ and a sequence of complex coefficients  $\{\lambda_j\}_{j\in\NN}$ such that
\begin{equation*}
f(x)=\sum_{j\in\NN} \lambda_j a_j(x),\qquad \sum_{j\in\NN} |\lambda_j|^p<\infty.
\end{equation*}

There are several possibilities to define equivalent "norms" in $H^p(\RR^d)$. We choose the atomic one, which is given by
\begin{equation*}
	\Vert f\Vert_{H^p(\RR^d)}=\inf \Big(\sum_{j\in\NN} |\lambda_j|^p\Big)^{\frac{1}{p}},
\end{equation*}
where the infimum is taken over all atomic decompositions of $f$. 

Throughout this paper let $X$ be a convex Lipschitz domain (by which we mean an open connected set with Lipschitz boundary; obviously the latter refers only to the case $d\geq 2$) in $\RR^d$ equipped with Lebesgue measure and the Euclidean metric. There is a number of possible definitions of $H^p$ spaces on subsets of $\RR^d$, see for instance the papers of Stein et al. \cite{Chang&Dafni&Stein_1999_TAMS, Chang&Krantz&Stein_1993_JFA} and of Miyachi \cite{Miyachi_1990_SM}. We choose the following one
\begin{equation*}
H^p(X)=\Big\{f: \exists F\in H^p(\RR^d)\ {\rm supp}\,F\subset \bar{X},\ F\big|_X=f \Big\}.
\end{equation*} 
We remark (see \cite[p.~137]{Stein_HarmonicAnalysis}) that each $f\in H^p(X)$ admits an atomic decomposition with all atoms supported in $X$, since $X$ is a Lipschitz domain. We set $\Vert f\Vert_{H^p(X)}$ similarly as in $\RR^d$. Observe that for $f$ and $F$ as above we have
\begin{equation}\label{eq:21}
\Vert F\Vert_{H^p(\RR^d)}\leq \Vert f\Vert_{H^p(X)},
\end{equation}
since for $F$ the underlying infimum is taken over a possibly larger set.

\subsection{Dual type spaces}

We need to give some meaning to the paring $\langle f,\fun\rangle$ for $\fun$ from a given orthonormal basis and $f\in H^p(\RR^d)$ or, more generally, for $f\in H^p(X)$. For this purpose we shall make use of the duality relation between $H^p(\RR^d)$ and $BMO(\RR^d)$ and Lipschitz spaces. 

A locally integrable function $f$ is in $BMO(\RR^d)$ (bounded mean oscillation space) if
\begin{equation*}
\Vert f\Vert_{BMO(\RR^d)}:=\sup_B \frac{1}{|B|} \int_B |f(x)-f_B|\,dx<\infty,
\end{equation*}
where the supremum is taken over all balls $B\subset \RR^d$ and $f_B=|B|^{-1} \int_B f$ is the mean value of $f$ over $B$. Observe that the expression above vanishes for constant functions. In fact, it is usual to define $BMO(\RR^d)$ as the quotient of the above space by the space of constant functions. Then $BMO(\RR^d)$ with the norm $\Vert \cdot\Vert_{BMO(\RR^d)}$ becomes a Banach space. For more details we refer to the literature, see \cite{Grafakos_Modern, Stein_HarmonicAnalysis}.

Now let $\Lambda_{\nu}(\RR^d)$, $\nu>0$, denote the Lipschitz space. If $\nu\notin\NN_+$, then $\Lambda_\nu(\RR^d)$ is composed of all functions $g\in\mathcal{C}^{(\lfloor \nu\rfloor)}(\RR^d)\cap L^\infty(\RR^d)$ satisfying the condition
\begin{equation*}
\Vert g\Vert_{\Lambda_\nu(\RR^d)}:= \Vert g\Vert_{L^\infty(\RR^d)} + \max_{\ven= \lfloor \nu\rfloor}\sup_{x,h\in \RR^d} \frac{\big|\partial^{n} g(x+h)-\partial^n g(x)\big|}{|h|^{\nu-\lfloor \nu\rfloor}}<\infty,
\end{equation*}
where $|h|$ denotes the Euclidean norm of the vector $h\neq 0$. If $\nu\in\NN_+$, then the above condition is replaced by
\begin{equation*}
\Vert g\Vert_{\Lambda_\nu(\RR^d)}:= \Vert g\Vert_{L^\infty(\RR^d)} +\max_{\ven= \nu-1} \sup_{x,h\in \RR^d} \frac{\big|\partial^n g(x+h)-2\partial^n g(x)+\partial^n g(x-h)\big|}{|h|}<\infty
\end{equation*}
for $g\in\mathcal{C}^{(\nu-1)}(\RR^d)\cap L^\infty(\RR^d)$. Finally, for $\nu=0$ we set $\Lambda_0(\RR^d):=BMO(\RR^d)$.

It is known that $BMO(\RR^d)$ is the dual of $H^1(\RR^d)$ (see \cite{Fefferman_1971_BAMS_duality,Fefferman&Stein_1972_AM}), whereas for $H^p(\RR^d)$, $p<1$, the duals are the Campanato spaces (see \cite{Campanato_1964_ASNSP(H^p-duality)} and for instance \cite[p.~55]{Lu_1995_FourLectures}). Nonetheless, the Lipschitz spaces described above and $H^p(\RR^d)$, $p\in(0,1]$, have a duality property too (see \cite{Garcia-Cuerva&RdFrancia_1985_N-HPA,Grafakos_Modern, Stein_HarmonicAnalysis,Uchiyama_Springer_2001}). Moreover, they are easier to handle and completely sufficient for our purposes.

The above-mentioned relation is the following: if $g\in \Lambda_{d(\frac{1}{p}-1)}(\RR^d)$, then
\begin{equation*}
|T_g(f)|:=\Big|\int_{\RR^d} g(x)f(x)\,dx\Big|\lesssim \Vert g\Vert_{\Lambda_{d (\frac{1}{p}-1) }(\RR^d)} \Vert f\Vert_{H^p(\RR^d)},
\end{equation*}
uniformly in $f\in H^p(\RR^d)\cap L^1(\RR^d)$. Moreover, since $H^p(\RR^d)\cap L^1(\RR^d)$ is dense in $H^p(\RR^d)$ (see for instance \cite[p.~54]{Lu_1995_FourLectures}), the functional $T_g$ has a unique bounded extension to the whole $H^p(\RR^d)$ with the same bound.

Now we will show a property of one-dimensional functions which are in $\Lambda_{\nu}(\RR)$, and then we will justify that the tensor products of such functions belong to $\Lambda_{\nu}(\RR^d)$. 

\begin{lm}\label{lm:deriv_in_Lipschitz}
	Let $\nu>1$ and $g\in\Lambda_{\nu}(\RR)$. Then, for $1\leq k\leq \lceil \nu\rceil-1$,  $g^{(k)}\in \Lambda_{\nu-k}(\RR)$ and
	\begin{equation*}
	\Vert g^{(k)}\Vert_{\Lambda_{\nu-k}(\RR)} \leq C_\nu \Vert g\Vert_{\Lambda_\nu(\RR)},\qquad g\in\Lambda_{\nu}(\RR),
	\end{equation*}
	for some positive constant $C_\nu$ independent of $g$.
\end{lm}
\begin{proof}	
	Fix non-zero $g\in\Lambda_{\nu}(\RR)$. Observe that it suffices to justify that $g^{(k)}$, $1\leq k\leq \lceil \nu\rceil-1$ are bounded, with the $L^\infty$-norm estimated by a constant times $\Vert g\Vert_{\Lambda_{\nu}(\RR)}$. Firstly we will check this for $k=\lceil \nu\rceil-1$.
	
	Let $\nu>1$ be such that $\nu\notin \NN$ (then $\lceil \nu\rceil-1=\lfloor \nu\rfloor$). We assume a contrario that there exists $u_0\in\RR$ such that $	|g^{(\lfloor \nu\rfloor)}(u_0)|> \big((2\lfloor \nu\rfloor)^{\lfloor \nu\rfloor} +2\big) \Vert g\Vert_{\Lambda_{\nu} (\RR)} $. Observe that for $h\in[-1,1]$ we have
	\begin{equation*}
	\big| g^{(\lfloor \nu\rfloor)}(u_0+h)-g^{(\lfloor \nu\rfloor)}(u_0)\big|\leq |h|^{\nu-\lfloor \nu\rfloor} \Vert g\Vert_{\Lambda_{\nu} (\RR)}\leq \Vert g\Vert_{\Lambda_{\nu} (\RR)}.
	\end{equation*}
	Hence,
	\begin{equation*}
		\big| g^{(\lfloor \nu\rfloor)}(u)\big|> \big((2\lfloor \nu\rfloor)^{\lfloor \nu\rfloor} +1\big) \Vert g\Vert_{\Lambda_{\nu} (\RR)},\qquad u\in[u_0-1,u_0+1].
	\end{equation*}
	Moreover, $g^{(\lfloor \nu\rfloor)}(x)$ does not change the sign on this interval since it is continuous. Now observe that
	\begin{align*}
	2^{\lfloor \nu\rfloor} \Vert g\Vert_{\Lambda_{\nu} (\RR)}&\geq \Big|\sum_{i=0}^{\lfloor \nu\rfloor}  {\lfloor \nu\rfloor \choose i} (-1)^{\lfloor \nu\rfloor-i} g\big(u_0+\frac{i}{\lfloor \nu\rfloor}\big) \Big|\\
	&=\Big| \int_0^{\lfloor \nu\rfloor^{-1}}\ldots \int_0^{\lfloor \nu\rfloor^{-1}}  g^{(\lfloor \nu\rfloor)}(u_0+s_1+\ldots+s_{\lfloor \nu\rfloor}) \,ds_{\lfloor \nu\rfloor}\ldots ds_1\Big|\\
	&=  \int_0^{\lfloor \nu\rfloor^{-1}}\ldots \int_0^{\lfloor \nu\rfloor^{-1}} \Big|  g^{(\lfloor \nu\rfloor)}(u_0+s_1+\ldots+s_{\lfloor \nu\rfloor})\Big| \,ds_{\lfloor \nu\rfloor}\ldots ds_1\\
	&> \big(2^{\lfloor \nu\rfloor} +\lfloor \nu\rfloor^{-\lfloor \nu\rfloor}\big) \Vert g\Vert_{\Lambda_{\nu} (\RR)}.
	\end{align*}
	The obtained contradiction proves that $g^{(\lfloor \nu\rfloor)}\in L^\infty(\RR)$ and, to be more precise, that $\Vert g^{(\lfloor \nu\rfloor)}\Vert_{L^\infty (\RR)}\leq \big((2k)^{\lfloor \nu\rfloor} +2\big) \Vert g\Vert_{\Lambda_{\nu} (\RR)}$.
	
	Observe that if $\nu\in\NN$, $\nu\geq 2$, then the proof, with minor modification, works as well. Hence, it suffices to justify that given $\ell\in\NN$ and $g\in\mathcal{C}^{\ell}(\RR)$ such that $g,g^{(\ell)}\in L^\infty(\RR)$ we have $g^{(j)}\in L^\infty(\RR)$  $j\in\{1,\ldots,\ell-1 \} $, with $\Vert g^{(j)}\Vert_{L^\infty(\RR)}\lesssim \Vert g\Vert_{L^\infty(\RR)}+\Vert g^{(\ell)}\Vert_{L^\infty(\RR)}$. This is an easy exercise but for the reader's convenience we give a short proof.
	
	Fix $\ell\in\NN$ and $j\in \{1,\ldots,\ell-1 \}$. Now assume a contrario that
	\begin{equation*}
	|g^{(j)}(u_0)|> j^j (2^{\ell} +1)\Vert g\Vert_{L^\infty(\RR)}+\Vert g^{(\ell)}\Vert_{L^\infty(\RR)}.
	\end{equation*}
	Notice that
	\begin{align*}
	&\hspace{-1cm}\Big|\sum_{i=0}^{\ell-j} {\ell -j\choose i} (-1)^{\ell -j-i} g^{(j)}(u_0+ih)\Big|\\
	&=\Big| \int_0^h\ldots \int_0^h g^{(\ell)}(u_0+s_1+\ldots+s_{\ell-j}) \, ds_{\ell-j}\ldots ds_1 \Big|\leq h^{\ell-j} \Vert g^{(\ell)}\Vert_{L^\infty(\RR)}. 
	\end{align*}
	Hence, for $h\in[-1,1]$ we obtain
	\begin{align*}
	\Big|\sum_{i=1}^{\ell-j} {\ell -j\choose i} (-1)^{\ell -j-i} g^{(j)}(x_0+ih)\Big|> j^j (2^{\ell} +1)\Vert g\Vert_{L^\infty(\RR)},
	\end{align*}
	and this sum does not change the sign in this interval. But on the other hand,
	\begin{align*}
	&\hspace{-1cm}(2^{\ell} +1)\Vert g\Vert_{L^\infty(\RR)}\\
	&< \Big| \int_0^{j^{-1}}\ldots \int_0^{j^{-1}} \sum_{i=1}^{\ell-j} {\ell -j\choose i} (-1)^{\ell -j-i} g^{(j)}(u_0+i(s_1+\ldots+s_j))\,ds_j\ldots ds_1\Big|\\
	&\leq  \sum_{i=1}^{\ell-j} {\ell -j\choose i} 2^{j} \Vert g\Vert_{L^\infty(\RR)}\\
	&\leq 2^{\ell} \Vert g\Vert_{L^\infty(\RR)}.
	\end{align*}
	This contradiction finishes the proof of the lemma.
\end{proof}

\begin{lm}\label{lm:tensor_in_Lambda_nu}
	Let $\nu>0$ and $g_1,\ldots, g_d\in \Lambda_\nu(\RR)$.  If $\nu=1$, then we additionally assume that $g_i'$, $i=1,\ldots,d$, exist and are bounded. Then the function 
	\begin{equation*}
	g(x)=g_1(x_1)\cdot\ldots\cdot g_d(x_d),\qquad x\in\RR^d,
	\end{equation*}
	belongs to $\Lambda_\nu(\RR^d)$ and
	\begin{equation*}
	\Vert g\Vert _{\Lambda_{\nu}(\RR^d)}\lesssim \prod_{ i=1}^d\Vert g_i\Vert _{\Lambda_{\nu}(\RR)}.
	\end{equation*}
\end{lm}
\begin{proof}
	Obviously $g\in L^\infty(\RR^d)$. Firstly assume that $\nu$ is a non-integer positive number. Let $n$ be a multi-index such that $\ven =\lfloor \nu\rfloor$. Then, for $h=(h_1,\ldots,h_d)\in\RR^d\setminus\{0\}$, we write the difference $\partial^n g(x+h)-\partial^n g(x)$ as
	\begin{align*}
	&\hspace{-1cm}\big(g_1^{(n_1)}(x_1+h_1)- g_1^{(n_1)}(x_1)\big)  g_2^{(n_2)}(x_2+h_2)\cdot\ldots\cdot g_d^{(n_d)}(x_d+h_d)\\
	&+ g_1^{(n_1)}(x_1)\big( g_2^{(n_2)}(x_2+h_2)- g_2^{(n_2)}(x_2)\big)  g_3^{(n_3)}(x_3+h_3)\cdot\ldots\cdot g_d^{(n_d)}(x_d+h_d)\\
	&+\ldots\\
	&+g_1^{(n_1)}(x_1) \cdot\ldots\cdot g_{d-1}^{(n_{d-1})}(x_{d-1})\big( g_d^{(n_d)}(x_d+h_d)- g_d^{(n_d)}(x_d)\big).
	\end{align*}
	Hence, by Lemma \ref{lm:deriv_in_Lipschitz} we get
	\begin{align*}
	\frac{\big|\partial^{n} g(x+h)-\partial^n g(x)\big|}{|h|^{\nu-\lfloor \nu\rfloor}}\lesssim \prod_{ i=1}^d\Vert g_i\Vert _{\Lambda_{\nu}(\RR)},\qquad x,h\in\RR^d,\ h\neq 0.
	\end{align*}
	
	Next, assume that $\nu\in\NN$ is such that $\nu\geq 2$. Then, for $\ven=\nu-1$, we estimate an expression of the form
	\begin{align*}
	&|h|^{-1}\big|\partial_{x_1}^{n_1} g_1(x_1+h_1)\cdot\ldots\cdot\partial_{x_d}^{n_d} g_1(x_d+h_d)-2\partial_{x_1}^{n_1} g_1(x_1)\cdot\ldots\cdot\partial_{x_d}^{n_d} g_d(x_d)\\
	&\qquad+\partial_{x_1}^{n_1} g_1(x_1-h_1)\cdot\ldots\cdot\partial_{x_d}^{n_d} g_1(x_d-h_d)\big|.
	\end{align*}
	Observe that if $n$ is a multi-index which has at least two non-zero components, then by Lemma \ref{lm:deriv_in_Lipschitz} the expression above is estimated by a constant times $\prod_{ i=1}^d\Vert g_i\Vert _{\Lambda_{\nu}(\RR)}$. Indeed, it easily follows from the mean value theorem, more precisely from the estimate
	\begin{equation*}
	\big|\partial_{x_i}^{n_i} g_i(x_i+h_i)-\partial_{x_i}^{n_i} g_i(x_i)\big| \leq  \big\Vert g_i^{(n_i+1)} \big\Vert_{L^\infty(\RR)} |h_i|\leq  \big\Vert g_i \big\Vert_{\Lambda_\nu(\RR)} |h_i|,
	\end{equation*}
	which holds since $n_i\leq \nu-2$. Otherwise, we can assume that $n=(\nu-1,0,\ldots,0)$. Hence, denoting $\bar{x}=(x_2,\ldots,x_d) $, $\bar{g}(\bar{x})=\prod_{i=2}^d g_i(x_i)$ and $\bar{h}=(h_2,\ldots,h_d)$, we write
	\begin{align*}
	&\partial_{x_1}^{\nu-1} g_1(x_1+h_1)\bar{g}(\bar{x}+\bar{h})-2\partial_{x_1}^{\nu-1} g_1(x_1)\bar{g}(\bar{x})+\partial_{x_1}^{\nu-1} g_1(x_1-h_1)\bar{g}(\bar{x}-\bar{h})\\
	&\qquad=\partial_{x_1}^{\nu-1} g_1(x_1+h_1)\big(\bar{g}(\bar{x}+\bar{h})-\bar{g}(\bar{x})\big)+\big(\partial_{x_1}^{\nu-1} g_1(x_1+h_1)-2\partial_{x_1}^{\nu-1} g_1(x_1)\\
	&\qquad+\partial_{x_1}^{\nu-1} g_1(x_1-h_1) \big)\bar{g}(\bar{x})-\partial_{x_1}^{\nu-1} g_1(x_1-h_1) \big(\bar{g}(\bar{x})-\bar{g}(\bar{x}-\bar{h})\big).
	\end{align*}
	Again, it suffices to use the mean value theorem and Lemma \ref{lm:deriv_in_Lipschitz} to get the required bound.
	
	Observe that this argument is valid also for $\nu=1$ provided we assume that $g'_i$ exist and are bounded. This finishes the proof.
\end{proof}

Now, we shall define the Lipschitz (and $BMO$) spaces in $X$ and prove similar duality. We say that a function $g$ defined on $X$ belongs to $\Lambda_\nu(X)$, $\nu\geq 0$, if there exists $G\in\Lambda_\nu(\RR^d)$ such that $G\big|_X=g$. Note that this type of definition differs from the one of $H^p(X)$, where we assumed that the extension vanishes outside $X$. In this case this is not possible because of the smoothness requirement. 

Moreover, we set
\begin{equation*}
\Vert g\Vert _{\Lambda_\nu(X)}=\inf \Vert G\Vert _{\Lambda_\nu(\RR^d)},
\end{equation*}
where the infimum is taken over all $G$ extending $g$ to $\RR^d$. With those definitions the following lemma holds.

\begin{lm}\label{lm:H^p/Lipschitz}
	Let $X$ be a convex Lipschitz domain in $\RR^d$. If $p\in(0,1]$ and $g\in \Lambda_{d(\frac{1}{p}-1)}(X)$, then
	\begin{equation*}
	|T_g(f)|:=\Big|\int_{X} g(x)f(x)\,dx\Big|\lesssim \Vert g\Vert_{\Lambda_{d (\frac{1}{p}-1) }(X)} \Vert f\Vert_{H^p(X)} ,
	\end{equation*}
	uniformly in $f\in H^p(X)\cap L^1(X)$. Consequently, the functional $T_g$ has a (unique) bounded extension to the whole $H^p(X)$ such that $|T_g(f)|\lesssim \Vert g\Vert_{\Lambda_{d(\frac{1}{p}-1)}(X)} \Vert f\Vert_{H^p(X)} $, $f\in H^p(X)$.
\end{lm}
\begin{proof}
	Indeed, we mentioned before that the claim is valid for $X=\RR^d$. Fix $p\in(0,1]$. For any $f\in H^p(X)\cap L^1(X)$ we take $F\in H^p(\RR^d)$ such that $F\big|_X=f$ and ${\rm supp}\,F\subset \bar{X}$, so that $F\in L^1(\RR^d)$. Similarly, for any $g\in\Lambda_{d (\frac{1}{p}-1) }(X)$ let $G\in\Lambda_{d (\frac{1}{p}-1) }(\RR^d)$ be an extension of $g$ to $\RR^d$. Then we have
	\begin{align*}
	\Big|\int_X f(x)g(x)\,dx\Big| =\Big|\int_{\RR^d} F(x)G(x)\,dx \Big| &\lesssim \Vert F\Vert_{H^p(\RR^d)} \Vert G\Vert_{\Lambda_{d (\frac{1}{p}-1) }(\RR^d)}\\
	&\leq \Vert f\Vert_{H^p(X)} \Vert G\Vert_{\Lambda_{d (\frac{1}{p}-1) }(\RR^d)},
	\end{align*}
	where in the last inequality we used \eqref{eq:21}.	By taking the infimum over $G$ we obtain the required bound. 
	
	Now we drop the assumption that $f\in L^1(X)$. For $\tilde{G}\in\Lambda_{d (\frac{1}{p}-1) }(\RR^d)$ let $\tilde{T}_{\tilde{G}}$ be the linear functional on $H^p(\RR^d)$ corresponding to $\tilde{G}$ so that there holds 
	\begin{equation*}
	|\tilde{T}_{\tilde{G}}(\tilde{F})|\lesssim \Vert \tilde{F}\Vert_{H^p(\RR^d)} \Vert \tilde{G}\Vert_{\Lambda_{d (\frac{1}{p}-1) }(\RR^d)},\qquad \tilde{F}\in H^p(\RR^d).
	\end{equation*} 
	We choose an extension $G$ of $g$ and define $T_g$ on the whole $H^p(X)$ by $T_g(f)=\tilde{T}_G(F)$ with the notation as above. Hence,
	\begin{equation*}
	|T_g(f)|\lesssim \Vert F\Vert_{H^p(\RR^d)} \Vert G\Vert_{\Lambda_{d (\frac{1}{p}-1) }(\RR^d)}\leq \Vert f\Vert_{H^p(X)} \Vert G\Vert_{\Lambda_{d (\frac{1}{p}-1) }(\RR^d)}.
	\end{equation*} 
	It suffices to take the infimum over $G$ to get the claim.
\end{proof}

One comment is in order here. Note that $T_g$ defined as in the proof of Lemma \ref{lm:H^p/Lipschitz} does not depend on the chosen extension $G$. Indeed, let $G_1$ and $G_2$ be extensions of $g$ to $\RR^d$. Let $F\in H^p(\RR^d)$ be supported in $\bar{X}$. Then there exists a sequence $F_k\in H^p(\RR^d)\cap L^1(\RR^d)$ such that $F_k\to F$ in $H^p(\RR^d)$ and $F_k$ are supported in $\bar{X}$ (see \cite[p.~109]{Stein_HarmonicAnalysis}). In fact, one can choose $F_k$ to be the partial sums of an atomic decomposition of $F$, where the atoms are supported in $X$. Such decomposition exists because $F\in H^p(\RR^d)$ and ${\rm supp}\,F\subset\bar{X}$, and therefore $F\big|_X\in H^p(X)$ which, in the light of the remark we made before, has such decomposition. Now fix $\ve>0$ and choose $N\in\NN$ so that 
\begin{equation*}
\Vert F-F_N\Vert_{H^p(\RR^d)}\leq \frac{\ve}{2\max(\Vert G_1\Vert_{\Lambda_{d (\frac{1}{p}-1) }(\RR^d)},\Vert G_2\Vert_{\Lambda_{d (\frac{1}{p}-1) }(\RR^d)})}.
\end{equation*}
Observe that
\begin{align*}
|T_{G_1}(F)-T_{G_2}(F)|&\leq |T_{G_1}(F_N)-T_{G_2}(F_N)|+|T_{G_1}(F-F_N)|+|T_{G_1}(F-F_N)|\\
&\leq \ve \big(\Vert G_1\Vert_{\Lambda_{d (\frac{1}{p}-1) }(\RR^d)}+\Vert G_2\Vert_{\Lambda_{d (\frac{1}{p}-1) }(\RR^d)} \big),
\end{align*}
since $T_{G_1}(F_N)=T_{G_2}(F_N)$ as $F_N\in L^1(X)$. This justifies that $T_g(f)$ does not depend on the chosen extension $G$. 

\subsection{Main theorem} 

Fix $p\in(0,1]$ and let $\{\fun\}_{n\in\NN^d}$, where $\fun\in \Lambda_{d(\frac{1}{p}-1)}(X)$, be an orthonormal basis in $L^2(X)$. We define the family of operators $\{R_r\}_{r\in(0,1)}$ via
\begin{equation}\label{R_def}
R_r f=\sum_{n\in\NN^d} r^{\ven} \langle f,\fun\rangle \fun,\qquad f\in L^1(X),
\end{equation}
where 
\begin{equation*}
\langle f,\fun\rangle=\int_X f(x)\overline{\fun(x)}\,dx. 
\end{equation*}
Note that the integral makes sense for $f\in L^1(X)$ since $\fun\in L^\infty(X)$ if $p<1$ and $\fun\in BMO(X)$ if $p=1$. We shall apply these operators to the elements of $H^p(X)$. For this purpose we need to give more general meaning to $\langle f,\fun\rangle$. Indeed, it can be defined by the means of Lemma \ref{lm:H^p/Lipschitz}, namely
\begin{equation*}
\langle f,\fun\rangle= T_{\fun}(f).
\end{equation*}
Recall that $T_{\fun}$ defined as in the proof of Lemma \ref{lm:H^p/Lipschitz} is unique (see the comment above). 

Let $R_r$, $r\in(0,1),$ be integral operators for which the associated kernels, denoted by $R_r(x,y)$, belong to $\mathcal{C}^P(X)$ (as functions of $x$, for any $y\in X$) for $P=\lfloor d(p^{-1}-1)\rfloor$, which means that all of their partial derivatives $\partial_{x_i}^j$, $j=0,\ldots,P$, exist and are continuous. Moreover, assume that $R_r(x,y)$ satisfy the following condition: there exist a constant $\gamma>0$ and a finite set $\Delta$ composed of positive numbers $\delta$ strictly greater than $d(p^{-1}-1)-P$, such that for each $k=0,\ldots,P$, there holds 
\begin{align}\label{cond:C}
\tag{C}
\begin{split}
\Big\Vert &R_r(x,\cdot)-\sum_{\ven\leq k} \frac{\partial^{n}_{x} R_r(x',\cdot)}{n_1!\cdot\ldots\cdot n_d!}\prod_{j=1}^d (x_j-x_j')^{n_j}\Big\Vert_{L^2(X)} \lesssim \sum_{\delta\in\Delta} (1-r)^{-(d+2k+2\delta)\gamma}|x-x'|^{k+\delta},
\end{split}
\end{align}
uniformly in $r\in(0,1)$ and $x,x'\in X$ such that $|x-x'|\leq 1/2$.
We emphasise that if $R_r(\cdot,y)$ are in $\mathcal{C}^{P+1}(X)$, then \eqref{cond:C} with $\Delta=\{1\}$ is implied by the easier estimate
\begin{equation*}
\sup_{x\in X} \Vert \partial^{n}_x R_r(x,\cdot)\Vert \lesssim (1-r)^{-(d+2\ven)\gamma},
\end{equation*}
uniformly in $r\in(0,1)$ and for $\ven\leq P+1$. Indeed, it suffices to use Taylor's theorem.

\begin{thm}\label{thm:Hp_general}
	Let $p\in(0,1]$, $s\in[p,2]$, and $X$ be a convex Lipschitz subset of $\RR^d$. Assume that the functions $\{\fun\}_{n\in\NN^d}$ belong to $\Lambda_{d(\frac{1}{p}-1)}(X)$, form an orthonormal basis in $L^2(X)$, and the associated kernels $R_r(x,y)$ satisfy condition \eqref{cond:C} with $\gamma>0$. Then the inequality 
	\begin{equation}\label{eq:0}
	\sum\limits_{n\in\NN^d}\frac{|\langle f,\fun\rangle|^s}{(\ven+1)^E}\lesssim \Vert f\Vert^s_{H^p(X)},
	\end{equation}
	holds uniformly in $f\in H^p(X)$, where
	\begin{equation}\label{Exponent_formula}
	E=\frac{(2-p)sd\gamma}{p}+\frac{(2-s)d}{2}.
	\end{equation}
\end{thm}

We remark that the above parameter $\gamma$ is not the same as $\gamma$ in \cite[Theorem~2.2]{Plewa_2018_sharpLaguerre_arxiv}; in fact if in the cited theorem $\mu$ is Lebesgue measure (and hence $N=d$), then both $\gamma$'s are equal up to the multiplicative constant $(d+2)$.

\begin{proof}
Fix $p$ and $s$ as in the claim. Firstly, we will prove the theorem for $(p,q)$-atoms, $q\in[2,\infty]$, and then we shall justify that it holds for all $f\in H^p(X)$. Let $a$ be a $(p,q)$-atom supported in a ball $B$ with the centre in $x'$. The following computation does not depend on $a$. Similarly as in \cite{Plewa_2018_sharpLaguerre_arxiv} and \cite{LiYuShi_2015_JFAA} in the first step we use an asymptotic estimate for the Beta function obtaining
\begin{align*}
\sum_{n\in\NN^d}\frac{\vert\langle  a, \fun\rangle\vert^s}{(\ven+1)^{E}}&\lesssim\sum_{n\in\NN^d}\int_0^1 r^{2\ven}(1-r)^{E-1}\vert\langle a, \fun\rangle\vert^s \,dr\\
&\leq\int_0^1 (1-r)^{E-1}\Big(\sum_{n\in\NN^d}r^{2\ven}\Big)^{\frac{2-s}{2}}\Big(\sum_{n\in\NN^d}\big(r^{s\ven}\vert\langle a,\fun\rangle\vert^s\big)^{\frac{2}{s}}\Big)^{\frac{s}{2}}\,dr\\
&\lesssim\int_0^1 (1-r)^{E-1}(1-r)^{-\frac{(2-s)d}{2}}\Vert R_r a\Vert^s_{L^2(X)}\,dr\\
&=\int_0^1 (1-r)^{\frac{(2-p)sd\gamma}{p}-1}\Vert R_r a\Vert^s_{L^2(X)}\,dr.
\end{align*}

Observe that
\begin{equation*}
\Vert R_r a\Vert^s_{L^2(X)}\leq  \Vert a\Vert^s_{L^2(X)}\leq |B|^{\big(\frac{1}{2}-\frac{1}{p}\big)s}.
\end{equation*}
Thus, the claim holds for $|B|\geq 1$. On the other hand, by \eqref{cond:C} we have
\begin{align*}
\Vert R_r a\Vert^s_{L^2(X)}&=\Big( \int_X \Big|\int_B R_r(x,y) a(x)\,dx\Big|^2 dy\Big)^{\frac{s}{2}}\\
&=\Big( \int_X \Big|\int_B \Big(R_r(x,y)-\sum_{\stackrel{i_1,\ldots,i_d\geq 0}{i_1+\ldots + i_d\leq P}} \frac{k!\, \partial^{i_1}_{x_1}\ldots \partial^{i_d}_{x_d} R_r(x',y)}{i_1! \cdot\ldots\cdot i_d!}\prod_{j=1}^d (x_j-x_j')^{i_j} \Big)\\
&\qquad \times a(x)\,dx\Big|^2 dy\Big)^{\frac{s}{2}}\\
&\lesssim \Big(\sum_{\delta\in\Delta} \int_B |a(x)| |x-x'|^{P+\delta} (1-r)^{-(d+2P+2\delta)\gamma} \,dx\Big)^s\\
&\lesssim \sum_{\delta\in\Delta}(1-r)^{-s(d+2P+2\delta)\gamma}|B|^{s(\frac{P+\delta}{d}+1-\frac{1}{p})}.
\end{align*}
Note that by the definition of $P$ and $\Delta$ there is $\frac{P+\delta}{d} +1-\frac{1}{p}> 0$ for every $\delta\in\Delta$. Hence,
\begin{align*}
\int_0^1 &\Vert R_r a\Vert^s_{L^2(X)}(1-r)^{\frac{(2-p)sd\gamma}{p}-1}dr\\
&\lesssim \sum_{\delta\in\Delta}\int_0^{1-|B|^{\frac{1}{2d\gamma}}} |B|^{s(\frac{P+\delta}{d}+1-\frac{1}{p})} (1-r)^{\frac{(2-p)sd\gamma}{p}-1-s\gamma(d+2P+2\delta)}dr\\
&+\int_{1-|B|^{\frac{1}{2d\gamma}}}^1 |B|^{s\big(\frac{1}{2}-\frac{1}{p}\big)}(1-r)^{\frac{(2-p)sd\gamma}{p}-1}dr\\
&\lesssim 1,
\end{align*}
uniformly in $B$ such that $|B|\leq 1$. This finishes the proof of the theorem for the atoms.

To complete the proof let us now justify that the claim holds for any $f\in H^p(X)$. Fix $f\in H^p(X)$ and its atomic decomposition $f=\sum_{j\in\NN} \lambda_j a_j$. Denote $f_J=\sum_{j=0}^J \lambda_j a_j$. We show that $\big\{ \{\langle f_J,\fun\rangle\}_{n\in\NN^d} \big\}_{J\in\NN}$ is a Cauchy sequence in $\ell^s\big( (\ven+1)^{-E}\big)$. Indeed, for $J>I$ we have for $s\in[p,1]$
\begin{align*}
\sum_{n\in\NN^d} \frac{|\langle f_J-f_I,\fun\rangle|^s}{(\ven+1)^E}\leq \sum_{j=I+1}^J |\lambda_j|^s\sum_{n\in\NN^d} \frac{|\langle a_j,\fun\rangle|^s}{(\ven+1)^E}\lesssim \sum_{j=I+1}^J |\lambda_j|^s\leq \Big(\sum_{j=I+1}^J |\lambda_j|^p\Big)^{s/p}.
\end{align*}
Since $\ell^s\big( (\ven+1)^{-E}\big)$ is complete with this metric we have shown that $\big\{ \{\langle f_J,\fun\rangle\}_{n\in\NN^d} \big\}_{J\in\NN}$ is a Cauchy sequence. Moreover, for $s\in[1,2]$ we use Minkowski's inequality and get
\begin{align*}
\Big(\sum_{n\in\NN^d} \frac{|\langle f_J-f_I,\fun\rangle|^s}{(\ven+1)^E}\Big)^{1/s}\leq \sum_{j=I+1}^J |\lambda_j|\Big(\sum_{n\in\NN^d} \frac{|\langle a_j,\fun\rangle|^s}{(\ven+1)^E}\Big)^{1/s}&\lesssim \sum_{j=I+1}^J |\lambda_j|\\
&\leq \Big(\sum_{j=I+1}^J |\lambda_j|^p\Big)^{1/p},
\end{align*}
and thus the considered sequence is a Cauchy sequence for this range of the parameter $s$ as well. Therefore, there exists $\{a_n\}_{n\in\NN^d}\in \ell^s \big( (\ven+1)^{-E}\big)$ such that
\begin{equation*}
\lim_{J\to\infty} \sum_{n\in\NN^d} \frac{|\langle f_J,\fun\rangle-a_n|^s}{(\ven+1)^E}=0.
\end{equation*}

We will justify that $a_n=\langle f,\fun\rangle$. The above equality yields
\begin{equation*}
\lim_{J\to\infty} \sum_{n\in\NN^d} \frac{|\langle f_J,\fun\rangle-a_n|^s}{(\ven+1)^{d+E}(1+\Vert \fun\Vert^s_{\Lambda_{d(\frac{1}{p}-1)}(X)}) }=0.
\end{equation*}
On the other hand, by Lemma \ref{lm:H^p/Lipschitz}
\begin{align*}
\lim_{J\to\infty} \sum_{n\in\NN^d} \frac{|\langle f_J-f,\fun\rangle|^s}{(\ven+1)^{d+E}(1+\Vert \fun\Vert_{\Lambda_{d(\frac{1}{p}-1)}(X)}) }\leq  \lim_{J\to\infty}\sum_{n\in\NN^d} \frac{\Vert f_J-f\Vert_{H^p(X)}^s \Vert \fun\Vert^s_{\Lambda_{d(\frac{1}{p}-1)}(X)}}{(\ven+1)^{d+E}(1+\Vert \fun\Vert^s_{\Lambda_{d(\frac{1}{p}-1)}(X)}) },
\end{align*}
and the latter limit is equal to zero. Hence, by the uniqueness of the limit we justified that $a_n=\langle f,\fun\rangle$.

Finally, fix $\ve>0$ and $J\in\NN$ such that $\Vert \langle f_J-f,\fun \rangle\Vert^s_{\ell^s((\ven+1)^{-E})}<\ve$. We estimate for $s\in[p,1]$
\begin{align*}
\sum_{n\in\NN^d} \frac{|\langle f,\fun\rangle|^s}{(\ven+1)^E}&\leq \sum_{n\in\NN^d} \frac{|\langle f-f_J,\fun\rangle|^s}{(\ven+1)^E}+\sum_{n\in\NN^d} \frac{|\langle f_J,\fun\rangle|^s}{(\ven+1)^E}\\
&\leq \ve + \sum_{j=0}^J |\lambda_j|^s \sum_{n\in\NN^d} \frac{|\langle a_j,\fun\rangle|^s}{(\ven+1)^E}\\
&\lesssim \ve + \Big(\sum_{j=0}^J |\lambda_j|^p\Big)^{s/p}\\
&\lesssim \ve + \Vert f\Vert^s_{H^p(X)}.
\end{align*}
If $s\in[1,2]$, then we proceed as before using Minkowski's inequality. This finishes the proof of the theorem.
\end{proof}

\subsection{Sharpness}
In this subsection we prove that the admissible exponent in Theorem \ref{thm:Hp_general} cannot be lowered, provided that we pose some additional assumptions on the basis $\{\fun\}_{n\in\NN^d}$. In fact, we focus only on the case $\fun(x)=\prod_{i=1}^d \funi(x_i)$. Therefore, we state our results in the one-dimensional situation and then make an appropriate remark on the general case $d\geq 1$.

We remark that although conditions \eqref{eq:17} and \eqref{eq:24} may look hard to meet, they turn out to be very natural in the classical orthonormal basis, such as Laguerre, Hermite, or Jacobi function expansions. 

Firstly, we construct a one-dimensional auxiliary atom $a$. Let $p\in (0,1]$, $P=\lfloor p^{-1}-1\rfloor$, $A\geq 1$ and $0<\delta\leq \frac{1}{2(P+1)}$. Consider the following function
\begin{equation}\label{eq:def_a}
a(u)=2^{-(P+2)} A^{1/p}\left\{
\begin{array}{ll}
-1, &u\in (0,\delta A^{-1}),\\
 C_j & u\in (j\delta A^{-1}, (j+1)\delta A^{-1}),\ j=1,\ldots,P,\\
 C_{P+1}& u\in ((P+1)\delta A^{-1}, A^{-1}),
\end{array}
\right.
\end{equation}
%\begin{equation}\label{eq:def_a}
%a(u)=\left\{
%\begin{array}{ll}
%-2^{-(P+2)}A^{1/p}, &u\in (0,\delta A^{-1}),\\
%2^{-(P+2)} C_1 A^{1/p} & u\in (\delta A^{-1}, 2\delta A^{-1})\\
%\vdots&\vdots\\
%2^{-(P+2)} C_P A^{1/p} & u\in (P\delta A^{-1}, (P+1)\delta A^{-1}),\\
%2^{-(P+2)} C_{P+1} A^{1/p} & u\in ((P+1)\delta A^{-1}, A^{-1}),
%\end{array}
%\right.
%\end{equation}
where $C_i$ are some constants to be determined. Note that if $|C_i|\leq 2^{P+2}$, then the estimate $\Vert a\Vert _{L^\infty}\leq |B|^{-1/p}$, where $B=(0,A^{-1})$, would follow.
Hence, if $C_i$ are satisfying this bound and are such that $\int u^k a(u)\,du=0$, $k=0,\ldots, P$, then $a$ is a $(p,\infty)$-atom. 

Observe that by the equality
\begin{equation*}
\int_{i\delta A^{-1}}^{(i+1)\delta A^{-1}} u^k\, du = \frac{1}{k+1}A^{-k-1} \delta^{k+1}((i+1)^{k+1}-i^{k+1}),\qquad k,i=0,\ldots, P,  
\end{equation*}
the cancellation properties come down to 
\begin{equation}\label{eq:34}
\sum_{i=1}^{P} C_i \delta^{k+1}((i+1)^{k+1} -i^{k+1})+C_{P+1}(1-((P+1)\delta)^{k+1})=\delta^{k+1},\qquad k=0,\ldots, P.
\end{equation}
This is a system of linear equations on $C_1,\ldots,C_{P+1}$ and one can solve it using Cramer's rule. A calculation shows that
\begin{equation*}
C_i=\sum_{\ell=0}^i {P+1 \choose \ell} (-1)^{\ell-1} \frac{1}{1-\ell\delta},\qquad i=1,\ldots,P+1. 
\end{equation*}
Indeed, inserting this into left hand side of \eqref{eq:34} we obtain
\begin{align*}
\delta^{k+1}\sum_{i=1}^{P} &\big((i+1)^{k+1}-i^{k+1}\big) \Big(\sum_{\ell=0}^i  {P+1 \choose \ell} (-1)^{\ell-1} \frac{1}{1-\ell\delta}\Big)\\
&\qquad +(1-((P+1)\delta)^{k+1})\Big(\sum_{\ell=0}^{P+1}  {P+1 \choose \ell} (-1)^{\ell-1} \frac{1}{1-\ell\delta}\Big)\\
&=\delta^{k+1} +\sum_{\ell=0}^{P+1} {P+1 \choose \ell} (-1)^{\ell-1} \frac{1-(\ell\delta)^{k+1}}{1-\ell\delta}\\
&=\delta^{k+1} +\sum_{j=0}^k \delta^j \sum_{\ell=0}^{P+1} {P+1 \choose \ell} (-1)^{\ell-1} \ell^j.
\end{align*}
Observe that for each $j$ the inner sum vanishes since $k\leq P$ and hence \eqref{eq:34} holds.

Now we clearly see that $|C_i|\leq 2^{P+2}$, $i\in\{1,\ldots,P+1\}$. Moreover, notice that
\begin{align*}
C_{P+1}= \sum_{\ell=0}^{P+1} {P+1 \choose \ell} (-1)^{\ell-1} \frac{1}{1-\ell\delta}=(-1)^{P} \int_{0}^1 (u^{-\delta}-1)^{P+1}\,du.
\end{align*}
Now, since
\begin{equation*}
(-\log u)\leq\frac{u^{-\delta}-1}{\delta}\leq (-\log u) u^{-1/(2P+2)},\qquad u\in(0,1),\ \delta\in (0,(2P+2)^{-1}),
\end{equation*}
it is easily seen that
\begin{equation}\label{eq:35}
|C_{P+1}|\simeq \delta^{P+1},\qquad \delta\in (0,(2P+2)^{-1}).
\end{equation}

To sum up, the function $a$ defined in \eqref{eq:def_a} is a $(p,\infty)$-atom.

\begin{prop}\label{prop:general_sharpness}
	Let the one-dimensional version of the assumptions of Theorem \ref{thm:Hp_general} be satisfied. Moreover, we assume that $(0,c)\subset X$ for some $c>0$ and that there exists $\tau> \frac{4\gamma-2p\gamma-p}{2p}$ such that for some $0<m\leq M$
	   \begin{equation}\label{eq:17}
	   m (k+1)^\tau u^{\frac{1+2\tau -2\gamma}{4\gamma}}\leq |\funk (u)|\leq M (k+1)^\tau u^{\frac{1+2\tau -2\gamma}{4\gamma}},
	   \end{equation}
	   uniformly in $u\in (0,cK^{-2\gamma})$, $k\leq K$ and $K\in\NN_+$, and $\funk(u)$ does not change the sign in this interval. Then the admissible exponent in \eqref{eq:0} cannot be lowered. 
\end{prop}

\begin{proof}
 	In order to prove this lemma we will construct an explicit sequence of atoms $a_K$, such that for $E$ defined in \eqref{Exponent_formula} and any $\ve>0$
 	\begin{equation}\label{eq:15}
 	\sum\limits_{k\in\NN}\frac{|\langle a_K,\funk\rangle|^s}{(k+1)^{E-\ve}}\gtrsim K^\ve,\qquad K\in\NN_+.
 	\end{equation}
 	 	
 	Let $K\in\NN_+$ and $a_K$ be an atom defined in \eqref{eq:def_a} with $A=K^{2\gamma}/c$ and some sufficiently small $\delta$. We will show that
\begin{equation}\label{eq:16}
|\langle a_K,\funk\rangle | \gtrsim (k+1)^\tau K^{2\gamma/p -1/2-\tau-\gamma}, \qquad 0\leq k\leq K.
\end{equation}
This suffices to prove \eqref{eq:15}. Indeed, since $\tau> (4\gamma-2p\gamma-p)/(2p)$, we have
\begin{align*}
\sum_{k\in\NN} \frac{|\langle a_K, \funk\rangle|^s}{(k+1)^{\frac{(2-p)s\gamma}{p}+\frac{s-2}{2}-\ve}}\gtrsim K^{2s\gamma/p -s/2-s\tau-s\gamma} \sum_{k=1}^K k^{\tau s-\frac{(2-p)s\gamma}{p}-\frac{2-s}{2}+\ve}\gtrsim K^\ve.
\end{align*}

Let us now justify \eqref{eq:16}. We have
\begin{equation*}
 \int_0^{cK^{-2\gamma}} a(u) \funk(u)\,du =\int_{(P+1)\delta cK^{-2\gamma}}^{cK^{-2\gamma}} a(u)\funk(u)\,du+\sum_{i=0}^P \int_{i\delta c K^{-2\gamma}}^{(i+1)\delta c K^{-2\gamma}} a(u)\funk (u)\,du.
\end{equation*}
Thus, the absolute value of the quantity above is bounded from below by
\begin{align*}
 2^{-(P+2)}&\Big(\frac{K^{2\gamma}}{c}\Big)^{\frac{1}{p} -\frac{1+2\tau+2\gamma}{4\gamma}}\frac{4M\gamma}{1+2\gamma+2\tau} (k+1)^\tau \Big(  |C_{P+1} |\frac{m}{M} \big(1 -((P+1)\delta)^{\frac{1+2\tau+2\gamma}{4\gamma}}\big)\\ &\qquad-\delta^{\frac{1+2\tau+2\gamma}{4\gamma}} \sum_{i=0}^P |C_i| \big((i+1)^{{\frac{1+2\tau+2\gamma}{4\gamma}}}-i^{\frac{1+2\tau+2\gamma}{4\gamma}}  \big)\Big)\\
&\gtrsim (k+1)^\tau K^{2\gamma/p -1/2-\tau-\gamma}\delta^{P+1}\Big( \frac{|C_{P+1}|}{\delta^{P+1}}  \frac{m}{M} \big(1 -((P+1)\delta)^{\frac{1+2\tau+2\gamma}{4\gamma}}\big)\\
&\qquad -\delta^{\frac{1+2\tau+2\gamma}{4\gamma}-(P+1)} \sum_{i=0}^P |C_i| \big((i+1)^{{\frac{1+2\tau+2\gamma}{4\gamma}}}-i^{\frac{1+2\tau+2\gamma}{4\gamma}}  \big)  \Big).
\end{align*}
Observe that taking $\delta$ sufficiently small, by the restraint for $\tau$ and \eqref{eq:35}, we obtain \eqref{eq:16}. 
\end{proof}

\begin{rem}\label{rem:sharpness}
	Notice that if $p^{-1}$ is not an integer and $\funk$ satisfy \eqref{eq:17} with $\tau=\frac{4\gamma-2p\gamma-p}{2p}$, then \eqref{eq:16} is also true. This implies for large $K$ the estimate
	\begin{equation*}
	\sum_{k\in\NN} \frac{|\langle a_K,\funk\rangle|}{(k+1)^E}\gtrsim \log K.
	\end{equation*}
	Hence, \eqref{eq:0} does not hold. But the case $p^{-1}\in\NN_+$ or $\tau>\frac{4\gamma-2p\gamma-p}{2p}$ Hardy's inequality may be valid, see Proposition \ref{prop:general_sharpness_2}.
\end{rem}

Sometimes condition \eqref{eq:17} holds with $\tau\leq \frac{4\gamma -2p\gamma-p}{2p}$ and hence Proposition \ref{prop:general_sharpness} cannot be applied in order to prove sharpness. However, estimate \eqref{eq:17} can be replaced by its analogue for the derivatives of $\funk$. We describe this situation in the following proposition.

\begin{prop}\label{prop:general_sharpness_2}
	Let the one-dimensional version of the assumptions of Theorem \ref{thm:Hp_general} be satisfied. Moreover, we assume that $(0,c)\subset X$ for some $c>0$, $\funk$ are $(P+1)$-times differentiable, where $P=\lfloor p^{-1}-1\rfloor$, and that there exists $\tau> (4\gamma-2p\gamma-p)/(2p)$ such that for some $0<m\leq M$ there holds
	\begin{equation}\label{eq:24}
	m (k+1)^{\tau} u^{\frac{1+2\tau -2\gamma}{4\gamma}-(P+1)}\leq |\funk^{(P+1)} (u)|\leq M (k+1)^{\tau} u^{\frac{1+2\tau -2\gamma}{4\gamma}-(P+1)},
	\end{equation}
	uniformly in $u\in (0,cK^{-2\gamma})$, $k\leq K$ and $K\in\NN_+$, and $\funk^{(P+1)}(u)$ does not change the sign in this interval. Then the admissible exponent in \eqref{eq:0} cannot be lowered. 
\end{prop}
\begin{proof}
	Fix $p\in(0,1]$ and set $P=\lfloor p^{-1}-1 \rfloor$. Let $K\in\NN$ and  $a_K$ be the same $H^p(X)$ atom as in Proposition \ref{prop:general_sharpness}. We show \eqref{eq:15}. Observe that denoting $A=K^{2\gamma}/c$ we have for some $\xi_u$ between $u$ and $(P+1)\delta/A$ the following equality 
	\begin{align*}
	\int_0^{A^{-1}} a(u) \funk(u) \,du &= \int_0^{A^{-1}} a(u)\Big( \funk(u)-\sum_{i=0}^P \frac{\funk^{(i)}\big(\frac{(P+1)\delta}{A}\big)}{i!} \big(u-\frac{(P+1)\delta}{A} \big) \Big)\,du\\
	&= \int_0^{A^{-1}} a(u) \frac{1}{(P+1)!} \funk^{(P+1)}(\xi_u) \Big( u-\frac{(P+1)\delta}{A}\Big)^{P+1} \,du.
	\end{align*}
	The absolute value of the latter integral can be estimated from below by
	\begin{align*}
	&\int_{\frac{(P+1)\delta}{A}}^{A^{-1}} 2^{-(P+2)} |C_{P+1}| A^{1/p} \frac{m}{(P+1)!} (k+1)^{\tau} \xi_u^{\frac{1+2\tau -2\gamma}{4\gamma}-(P+1)}  \Big( u-\frac{(P+1)\delta}{A}\Big)^{P+1}\,du\\
	&\qquad-\int_0^{\frac{(P+1)\delta}{A}} A^{1/p} \frac{M}{(P+1)!} (k+1)^{\tau}\xi_u^{\frac{1+2\tau -2\gamma}{4\gamma}-(P+1)} \Big( \frac{(P+1)\delta}{A}-u\Big)^{P+1}\,du\\
	&\geq \frac{M}{(P+1)!} A^{1/p} (k+1)^{\tau} \bigg( 2^{-(P+2)} |C_{P+1}| \frac{m}{M} \Big(  \frac{(P+1)\delta}{A}\Big)^{\frac{1+2\tau -2\gamma}{4\gamma}-(P+1)}\\
	&\qquad\times \int_{\frac{(P+1)\delta}{A}}^{A^{-1}}  \Big( u-\frac{(P+1)\delta}{A}\Big)^{P+1}\,du - \int_0^{\frac{(P+1)\delta}{A}} u^{\frac{1+2\tau -2\gamma}{4\gamma}-(P+1)} \Big( \frac{(P+1)\delta}{A}-u\Big)^{P+1}\, du\bigg)\\
	&= \frac{M}{(P+2)!} ((P+1)\delta)^{\frac{1+2\tau -2\gamma}{4\gamma}} (k+1)^{\tau} A^{\frac{1}{p} -\frac{1+2\tau -2\gamma}{4\gamma}}\\
	&\qquad\times \Big( 2^{-(P+2)} |C_{P+1}| \frac{m}{M} \big(1-(P+1)\delta \big)^{P+2}-\big( (P+1)\delta \big)^{P+2}   \Big)\\
	&\gtrsim  (k+1)^{\tau} A^{\frac{1}{p} -\frac{1+2\tau -2\gamma}{4\gamma}},
	\end{align*}
	for $\delta$ sufficiently small, since we have \eqref{eq:35}.
	
	Hence, we obtained
	\begin{equation*}
	|\langle a_K,\funk\rangle| \gtrsim (k+1)^{\tau} K^{1/p-\tau-\gamma-1/2},
	\end{equation*}
	uniformly in $k\in\NN_+$ and $k\leq K$. Thus, for any $\ve>0$ 
	\begin{align*}
	\sum_{k=0}^\infty \frac{|\langle a_K,\funk \rangle|^s}{(k+1)^{E-\ve}}\gtrsim \sum_{k=0}^K \frac{(k+1)^{s\tau} K^{s/p-s\tau-s\gamma-s/2}}{(k+1)^{E-\ve}}\simeq K^\ve,\qquad K\in\NN_+,
	\end{align*}
	since $\tau$ is large enough. This finishes the proof of the lemma. 
\end{proof}

\begin{rem}\label{rem:sharpness_multidim}
	In the multi-dimensional case the construction also works if the functions in the considered orthonormal basis are products of one-dimensional functions for which properties \eqref{eq:17} or \eqref{eq:24} holds. Indeed, denoting $A_K(x)=\prod_{i=1}^d a_K(x_i)$ we have by \eqref{eq:16} the following lower bound
	\begin{align*}
	\sum_{n\in\NN} \frac{|\langle A_K,\fun\rangle|^s}{(\ven+1)^E}\gtrsim  K^{sd(\frac{2\gamma}{p} -\frac{1}{2}-\tau-\gamma)} \sum_{K/2\leq n_i\leq K}\frac{(n_i+1)^{s\tau}}{(\ven+1)^E}\gtrsim K^{sd(\frac{2\gamma}{p} -\frac{1}{2}-\tau-\gamma)-E+sd\tau +d}=K^\ve,
	\end{align*}
	uniformly in large $K$. 
\end{rem}

\begin{rem}
	Notice that \eqref{eq:15}, generalized to the multi-dimensional situation, and the uniform boundedness principle (in a stronger version than usual, see for instance \cite[Theorem~2.5]{Rudin_1991_FunctionalAnalysis}) imply that there exists $f\in H^p(X)$ such that
	\begin{equation*}
	\sum\limits_{n\in\NN^d}\frac{|\langle f,\fun\rangle|^s}{(\ven+1)^E}=\infty.
	\end{equation*}
	This is consistent with what was proved in author's articles concerning Hardy's inequality on $H^1$, see \cite{Plewa_2018_sharpLaguerre_arxiv,Plewa_2019_JAT,Plewa_2020_TJM}.
\end{rem}

\section{Laguerre standard functions}\label{S:Laguerre_standard}

The {\it standard Laguerre functions} $\{\Lfunk\}_{k\in\mathbb{N}}$ of order $\al>-1$ are defined on $\RR_+$ by 
\begin{equation}\label{eq:25}
\Lfunk(u)= \Big(\frac{\Gamma(k+1)}{\Gamma(k+\al+1)} \Big)^{1/2} L_k^\al (u) e^{-u/2} u^{\al/2},\qquad u>0,
\end{equation}
where $L_k^\al(u)$ are the Laguerre polynomials (see \cite{Szego1959}).
Moreover, in the multi-dimensional case $\Lfun(x)$ are defined as the tensor product of the one-dimensional functions, namely
\begin{equation*}
\Lfun(x)=\prod_{i=1}^d \Lfuni(x_i),\qquad x=(x_1,\ldots, x_d)\in\RR_+^d;
\end{equation*}
here $\al=(\al_1,\ldots,\al_d)\in(-1,\infty)^d$ and $n=(n_1,\ldots, n_d)\in\NN^d$. The system $\{\Lfun\}_{n\in\NN^d}$ forms an orthonormal basis in $L^2(\RR_+^d,\,dx)$.
The following estimates are known for the one-dimensional standard Laguerre functions (see \cite[p.~435]{Muckenhoupt_1970} and \cite[p.~699]{Askey_Wainger_1965_AJM})
\begin{equation}\label{eq:38}
\begin{split}
\vert \Lfunk(u)\vert\lesssim\left\{ \begin{array}{ll}
(uk')^{\al/2},  & 0<u\leq 1/k',\\
(uk')^{-1/4}, & 1/k'<u\leq k'/2,\\
(k'(k'^{1/3}+\vert u- k'\vert))^{-1/4}, & k'/2<u\leq (3k')/2,\\
\exp(-\gamma u), & 3k'/2<u<\infty,
\end{array}\right.
\end{split}
\end{equation}
where $k'=\max(4k+2\al+2,2)$ and $\gamma>0$ depends only on $\alpha$. 

These estimates imply for all $\al\geq 0$ the bound (cf. \cite[p.~94]{Stempak_1994_Tohoku}),
\begin{equation*}
\Vert \Lfunk\Vert_{L^\infty(\RR_+)} \lesssim 1,\qquad k\in\NN.
\end{equation*}
Moreover, using the formula (see \cite[p.~95]{Stempak_1994_Tohoku})
\begin{equation}\label{eq:19}
	(\Lfunk)'(u)=-k^{1/2}u^{-1/2}\mathcal{L}_{k-1}^{\al+1}(u)+\frac{1}{2}\Big(\frac{\alpha}{u}-1\Big)\Lfunk(u), 
\end{equation}
where $\mathcal{L}_{-1}^{\al+1}\equiv 0$, for $\al\in \{0\}\cup[2,\infty)$ we obtain
\begin{equation*}
\Vert (\Lfunk)'\Vert_{L^\infty(\RR_+)} \lesssim k+1, \quad k\in\NN.
\end{equation*}
More generally, for $j\in\NN$ and $\al\in\{0,2,\ldots,2j\}\cup (2j,\infty)$ there holds (see \cite[Lemma~1]{Satake_2000_JMSJ})
\begin{equation}\label{eq:1}
\Vert (\Lfunk)^{(j)}\Vert_{L^{\infty}(\RR_+)}\lesssim (k+1)^{j}, \qquad k\in\NN.
\end{equation}

Now we will justify that $\Lfun$ belong to the spaces $\Lambda_\nu(\RR^d_+)$. For that purpose we will indicate an extension $\tilde{\mathcal{L}}_n^\al\in\Lambda_{\nu}(\RR^d)$ of $\Lfun$ to $\RR^d$. Following the idea used in \cite[p.~94]{Shi&Li_2016_JMSJ} in the case $d=1$ we define
\begin{equation*}
\tilde{\mathcal{L}}_n^\al(x)=\prod_{i=1}^d \tilde{\mathcal{L}}_{n_i}^{\al_i}(x_i),
\end{equation*}
where, for $\al_i$ which is not an even integer,
\begin{equation*}
\tilde{\mathcal{L}}_{n_i}^{\al_i}(x_i)=\left\{\begin{array}{ll}
\Lfuni(x_i),& x_i>0\\
0, & x_i\leq 0,
\end{array}\right.
\end{equation*}
and, for $\al_i$ which is an even integer,
\begin{equation*}
\tilde{\mathcal{L}}_{n_i}^{\al_i}(x_i)=\psi(n_ix_i) \Lfuni(x_i),\qquad x_i\in\RR.
\end{equation*}
In the latter case the definition of $\Lfuni$ is naturally extended by the initial formula \eqref{eq:25} to the whole real line, and $\psi$ is a smooth function supported in $[-1,\infty)$ such that $\psi\equiv 0$ on $\RR_+$ and $\Vert \psi^{(j)}\Vert_{L^\infty(\RR)}\lesssim 1$, $j\in\NN$. For an example of such function see \cite{Shi&Li_2016_JMSJ}. 

%\begin{equation*}%\label{eq:psi_def}
%\psi(u)=\left\{\begin{array}{ll}
%1,& u>0\\
%(1-e^{1/u}) \exp \big(-\frac{e^{1/u}}{u+1} \big), & -1<u\leq 0,\\
%0, & x\leq -1.
%\end{array}\right.
%\end{equation*}

In view of \cite[Corollary~2.4]{Shi&Li_2016_JMSJ} we see that given $\nu>0$ we have $\tilde{\mathcal{L}}_{n_i}^{\al_i}\in \Lambda_{\nu}(\RR)$ for $\al_i\in[2\nu,\infty)$. Secondly, if $\al_i$ is an even integer, then $\tilde{\mathcal{L}}_{n_i}^{\al_i}\in \Lambda_{\nu}(\RR)$ for all $\nu> 0$. And lastly, for $\al\in[0,\infty)$ the functions $\Lfuni$ are bounded and hence are in $BMO(\RR)$. Thus, by Lemma \ref{lm:tensor_in_Lambda_nu} if $p\in(0,1]$ and $\al\in\big(\{0,2,\ldots,2P\}\cup [2d(p^{-1}-1),\infty)\big)^d$, where $P=\lfloor d(\frac{1}{p}-1)\rfloor$, then $\tilde{\mathcal{L}}_{n}^{\al}\in\Lambda_{d(\frac{1}{p}-1)}(\RR^d)$, and therefore $\Lfun\in\Lambda_{d(\frac{1}{p}-1)}(\RR^d_+)$. In order to satisfy the additional assumption in Lemma \ref{lm:tensor_in_Lambda_nu}, we have used the fact that for $\al\in\{0\}\cup[2,\infty)$ the functions $(\Lfunk)'$ exist and are bounded. 

The family of operators $\{\Rop\}$ associated with $\{\Lfun\}_{n\in\NN^d}$ and given by
\begin{equation*}
\Rop f=\sum_{n\in\NN^d} r^{\ven} \langle f,\Lfun\rangle \Lfun, \qquad r\in(0,1),
\end{equation*}
is composed of integral operators, with the kernels of the form
\begin{equation*}
\Rop(x,y)=\sum_{n\in\NN^d} r^{\ven} \Lfun(x)\Lfun(y).
\end{equation*}
It can be explicitly written as the product of the kernels $\Ropi(x_i,y_i)$ (cf. \cite{Plewa_2018_sharpLaguerre_arxiv,Szego1959})
\begin{equation*}
\Ropi (x_i,y_i)(1-r)^{-1}r^{-\al_i/2}\exp\Big(-\frac{1}{2}\frac{1+r}{1-r}(x_i+y_i)\Big)I_{\al_i }\Big(\frac{2r^{1/2}}{1-r}\sqrt{x_i y_i}\Big),
\end{equation*}
where $I_s (u)$ denotes the Bessel function of the first kind and order $s$. It is a real, positive, and smooth function for $s>-1$. 

In fact, we do not need this explicit formula for $\Rop(x,y)$ to prove Hardy's inequality. However, for the completeness of the presentation we gave it above. On the other hand, its analogue for Laguerre functions of Hermite type will be of paramount importance.

Now we are ready to verify condition \eqref{cond:C} for the standard Laguerre functions.

\begin{lm}\label{lm:Laguerre_standard_R_deriv}
	For $j\in\NN$ and $\al\in\{0,2,\ldots,2j \}\cup (2j,\infty)$ there holds
\begin{equation*}
\sup_{u>0} \Big\Vert \partial^j_u \Rop(u,\cdot)\Big\Vert_{L^2(\RR_+)}\lesssim (1-r)^{-\frac{1+2j}{2}},\qquad r\in(0,1).
\end{equation*}
\end{lm}
\begin{proof}
We simply apply Parseval's identity and \eqref{eq:1} obtaining
\begin{align*}
\sup_{u>0} \Big\Vert \partial^j_u \Rop(u,\cdot)\Big\Vert_{L^2(\RR_+)}\leq \Big( \sum_{k\in \NN} r^{2k} \big\Vert(\Lfunk)^{(j)}\big\Vert_{L^2(\RR_+)}^2 \Big)^{1/2}\lesssim (1-r)^{-\frac{1+2j}{2}},
\end{align*}
uniformly in $r\in(0,1)$. Notice that interchanging differentiation with summation is possible due to polynomial growth on $\Vert (\Lfunk)^{(i)}\Vert_{L^\infty(\RR_+)}$, $0\leq i\leq j$ (see \eqref{eq:1}), and the Lebesgue dominated convergence theorem. Analogous remarks apply to similar operations in this and the next sections.
\end{proof}

\begin{lm}\label{lm:Laguerre_standard_R_diff}
	Let $j\in\NN$ and $\al\in(2j,2j+2)$. Then the estimate
	\begin{equation*}
	\Big\Vert \partial^{j}_u \Rop(u,\cdot)-\partial^{j}_u \Rop(u',\cdot)\Big\Vert_{L^2(\RR_+)}\lesssim (1-r)^{-(1+\al)/2}|u-u'|^{\al/2-j},
	\end{equation*}
	holds uniformly in $r\in(0,1)$ and $u,u'>0$. 
\end{lm}
\begin{proof}
	Fix $j\in\NN$. By \cite[Lemma~2.2]{Shi&Li_2016_JMSJ} we have for $\al\in(2j,2j+2)$ the estimate
	\begin{equation*}
	\big|(\Lfunk)^{(j)}(u)-(\Lfunk)^{(j)}(u')\big|\lesssim (k+1)^{\al/2} |u-u'|^{\al/2-j}, \qquad u,u'>0,\ k\in\NN. 
	\end{equation*}
	Hence,  Parseval's identity implies
	\begin{align*}
	\Big\Vert \partial^{j}_u \Rop(u,\cdot)-\partial^{j}_u \Rop(u',\cdot)\Big\Vert_{L^2(\RR_+)}&\leq \Big(\sum_{k\in\NN} r^{2k} (k+1)^{\al}  \Big)^{1/2}|u-u'|^{\al/2-j},
	\end{align*}
	uniformly in $u,u'\in\RR_+$, and the claim follows by simple estimate of the latter series (cf. \cite[(3.3)]{Plewa_2019_JAT}).
\end{proof}

Now we easily obtain the following proposition.

\begin{prop}\label{prop:standard_Laguerre_cond_C}
	If $k\in\NN$ and $\al\in\big(\{0,2,\ldots,2k \}\cup (2k,\infty)\big)^d$, then
	\begin{align*}
	\begin{split}
	\Big\Vert &\Rop(x,\cdot)-\sum_{\ven\leq k} \frac{ \partial^{n}_{x} \Rop(x',\cdot)}{n_1! \cdot\ldots\cdot n_d!}\prod_{i=1}^d (x_i-x_i')^{n_i}\Big\Vert_{L^2(\RR^d_+)}\lesssim \sum_{\delta\in\Delta^\al_k}(1-r)^{-\frac{d+2k+2\delta}{2}}|x-x'|^{k+\delta},
	\end{split}
	\end{align*}		
	uniformly in $r\in(0,1)$ and $x,x'\in\RR^d_+$, where
	\begin{equation*}
	\Delta^\al_k=\{1\}\cup\{\al_i/2-k:\ \al_i\in(2k,2k+2) \}.
	\end{equation*}
\end{prop}
\begin{proof}
Fix $\al\in\big(\{0,2,\ldots,2k-2\}\cup [2k,\infty)\big)^d$. If for all $i=1,\ldots,d$ there is $\al_i\notin (2k,2k+2)$, then apply Taylor's theorem with the reminder of $(k+1)$-th order, and Lemma \ref{lm:Laguerre_standard_R_deriv} with $j\leq k+1$. On the other hand, if some $\al_i\in(2k,2k+2)$, then proceed as before but with $k$-th order reminder, obtaining
\begin{align*}
\sum_{\ven=k} \frac{k!}{n_1! \cdot\ldots\cdot n_d!} \prod_{i=1}^d  \Big( \big(\partial_{x_i}^{n_i}\Ropi(\xi_i,y_i)-\partial_{x_i}^{n_i}\Ropi(x'_i,y_i)\big)(x-x_i')^{n_i}\Big),
\end{align*}
where for every $i\in\{1,\ldots,d \}$ the number $\xi_i$ lies between $x_i$ and $x'_i$. Now for each difference above we apply Lemma \ref{lm:Laguerre_standard_R_diff} if $\al_i\in(2n_i,2n_i+2)$,  or the mean value theorem and Lemma \ref{lm:Laguerre_standard_R_deriv} in the opposite situation.
\end{proof}

Although the following lemma will be applied strictly to prove sharpness of Hardy's inequality associated with the standard Laguerre expansions, we stress that this is an interesting result and possibly it could be widely used in other problems concerning the functions $\Lfunk$.

Here and later on we use the convention that $A\simeq -B$ for positive $B$ means that $A$ is negative and $(-A)\simeq B$.

\begin{lm}\label{lm:standard_Laguerre_sharp_estimates}
	Let $\al\geq 0$ and $j,\ell\in\NN$ be given. There exists a constant $c>0$ such that if $\ell\geq j$, then there holds
	\begin{equation*}
	\frac{d^j}{d u^j}\frac{\Lfunk(u)}{u^{\al/2-\ell}}\simeq
	(k+1)^{\al/2} u^{\ell-j},
	\end{equation*}
	whereas if $\ell\leq j$, then
	\begin{equation*}
	\frac{d^j}{d u^j}\frac{\Lfunk(u)}{u^{\al/2-\ell}}\simeq 
	(-1)^{j-\ell} (k+1)^{\al/2+j-\ell},
	\end{equation*}
	uniformly in $k\in\NN$ and $u\in(0,c(k+1)^{-1})$. 
\end{lm}
\begin{proof}
	We will apply the induction over $j$ separately in both cases. Note that the claim holds for $j=0$ and any 
	$\ell\in\NN$ (this is a known result, see \cite[pp.~435,~453]{Muckenhoupt_1970}). We assume that it is valid for some $j$ and we will justify it for $j+1$. 
	Observe that by \eqref{eq:19} we have
	\begin{equation}\label{eq:20}
		\frac{d^{j+1}}{d u^{j+1}}\frac{\Lfunk(u)}{u^{\al/2-\ell}}=\frac{d^j}{d u^j} \Big(\ell  \frac{\Lfunk(u)}{u^{\al/2-\ell+1}} -\frac{1}{2}  \frac{\Lfunk(u)}{u^{\al/2-\ell}} -\sqrt{k} \frac{\mathcal{L}_{k-1}^{\al+1}(u)}{u^{(\al+1)/2-\ell}}\Big).
	\end{equation} 
	Notice that if $\ell\geq j+1$, then the components on the right hand side of \eqref{eq:20} are of the sizes: $(k+1)^{\al/2} u^{\ell-j-1}$, $(k+1)^{\al/2} u^{\ell-j}$, and $(k+1)^{\al/2 +1} u^{\ell-j}$, respectively, and the first one is the dominating.
	
	It remains to justify the case $ j\geq \ell$. Let us assume that for some such $j$ the estimate holds. Then we have similarly as above. The second and the third summand on the right hand side of \eqref{eq:20} are of the sizes (and signs): $(-1)^{j-\ell+1} (k+1)^{\al/2+j-\ell}$ and $(-1)^{j-\ell+1} (k+1)^{\al/2+1+j-\ell}$, respectively. On the other hand, the first component we decompose and get
	\begin{equation*}
	\frac{d^j}{d u^j}\ell  \frac{\Lfunk(u)}{u^{\al/2-\ell+1}}= \ell\frac{d^{j-1}}{d u^{j-1}} \Big((\ell-1)  \frac{\Lfunk(u)}{u^{\al/2-\ell+2}} -\frac{1}{2}  \frac{\Lfunk(u)}{u^{\al/2-\ell}} -\sqrt{k} \frac{\mathcal{L}_{k-1}^{\al+1}(u)}{u^{(\al+1)/2-\ell}}\Big).
	\end{equation*}
	Again, the first summand can be decomposed, and the two remaining are of the same size (and sign) as before. Moreover, note that the $i$-th decomposition of the first resulting component brings the multiplicative constant $\ell-i+1$. But this proves that the component vanishes, since $j\geq \ell$. Hence, in this case \eqref{eq:20} is of the size and sign $(-1)^{j-\ell+1} (k+1)^{\al/2+1+j-\ell}$. This finishes the proof of the lemma.
\end{proof}

We now are ready to prove Hardy's inequality associated with the standard Laguerre functions.

\begin{thm}\label{thm:main_Laguerre_standard}
	Let $p\in(0,1)$, $s\in[p,2]$, and denote $P:=\lfloor d(p^{-1}-1)\rfloor$. For 
	\begin{equation*}
	\al\in\big(\{0,2,\ldots,2P \}\cup (2d(p^{-1}-1),\infty)\big)^d
	\end{equation*}
	there holds
	\begin{equation*}
	\sum_{n\in\NN^d}\frac{|\langle f,\Lfun\rangle|^s}{(\ven+1)^E}\lesssim \Vert f\Vert^s_{H^p(\RR^d_+)}, \qquad f\in H^p(\RR^d_+),
	\end{equation*}
	where $E= d+sd\big(p^{-1}-1\big)$, and the exponent is sharp.
\end{thm}
\begin{proof}
	Proposition \ref{prop:standard_Laguerre_cond_C} ensures that the appropriate version of \eqref{cond:C} holds for the standard Laguerre functions, and hence by Theorem \ref{thm:Hp_general} we obtain associated Hardy's inequality.
	
	Observe that if  $\al\in\big(\{0,2,\ldots,2P\}\cup (2d(p^{-1}-1),\infty)\big)^d$ and for some $i\in\{1,\ldots,d\}$ there is $\al_i/2=d(p^{-1}-1)$, then, although $\Lfun\in\Lambda_{d (\frac{1}{p}-1) }(\RR^d)$, our method does not give Hardy's inequality, unless $d(p^{-1}-1)$ is an integer. Indeed, in such case there exists $\delta= d(p^{-1}-1)-\lfloor d(p^{-1}-1)\rfloor$ in the appropriate version of \eqref{cond:C}, for which the reasoning is not valid. However, due to Remark \ref{rem:sharpness} we see that in such case Hardy's inequality does not hold with the exponent given by \eqref{Exponent_formula}. This agrees with the already known results concerning this topic, see \cite{Satake_2000_JMSJ,Shi&Li_2016_JMSJ}.
	 
	On the other hand, by Lemma \ref{lm:standard_Laguerre_sharp_estimates} (with $j=\ell=0$) we have
	\begin{equation}\label{eq:22}
	A((k+1)u)^{\al/2}\leq \Lfunk(u)\leq B((k+1)u)^{\al/2},\qquad 0<u\leq \frac{c}{k+1},
	\end{equation}
	where $c>0$, and we see that, since $\gamma=1/2$, condition \eqref{eq:17} holds for $\{\Lfunk\}_{k\in\NN}$ with $\tau=\al/2$. Hence, by Proposition \ref{prop:general_sharpness} sharpness follows for $\al> 2(p^{-1}-1)$ (if $d=1$; in general for $\al\in \big(2d(p^{-1}-1),\infty)^d$). Moreover, if $\al$ is an even integer smaller that $2(p^{-1}-1)$, then we apply Proposition \ref{prop:general_sharpness_2} and Lemma \ref{lm:standard_Laguerre_sharp_estimates} (with $\ell=\al/2$ and $j=P+1$).
	The reasoning can be transferred to the multi-dimensional situation, see Remark \ref{rem:sharpness_multidim}.
	
\end{proof}

\section{Laguerre functions of Hermite type}\label{S:Laguerre_Hermite}

The {\it Laguerre functions of Hermite type} $\hfunk$, $k\in\NN$, are defined by the following relation with the standard Laguerre functions
\begin{equation}\label{eq:33}
\hfunk(u)=\sqrt{2u}\Lfunk(u^2)=\Big(\frac{2\Gamma(k+1)}{\Gamma(k+\al+1)} \Big)^{1/2} L_k^\al (u^2) e^{-u^2/2} u^{\al+1/2},
\end{equation}
where $u>0$ and $\al>-1$. In the multi-dimensional situation $\hfun(x)$ are defined as the tensor products of $\hfuni(x_i)$. The system $\{\hfun\}_{n\in\NN^d}$ is then an orthonormal basis in $L^2(\RR^d_+)$.

The functions $\funk$ are bounded on $\RR_+$ for $\al\geq-1/2$. Moreover,
\begin{equation}\label{eq:9}
\Vert \hfunk\Vert_{L^{\infty}(\RR_+)}\lesssim (k+1)^{-1/12},\qquad k\in\NN.
\end{equation}
The following recurrent formula for the derivatives of $\hfunk$ holds (see \cite[p.~100]{Stempak_1994_Tohoku})
\begin{equation}\label{eq:6}
(\hfunk)'(u)=-2\sqrt{k}\varphi_{k-1}^{\al+1}(u)+\left(\frac{2\al+1}{2u}-u\right)\hfunk(u),
\end{equation}
where $\varphi_{-1}^{\al+1}\equiv 0 $. Hence, for $\al\in\{-1/2\}\cup[1/2,\infty)$, using \eqref{eq:33} and \eqref{eq:38} one obtains
\begin{equation*}
\left\Vert (\hfunk)'\right\Vert_{L^{\infty}(\RR_+)}\lesssim (k+1)^{5/12},\qquad k\in\NN.
\end{equation*}
For the boundedness of higher order derivatives see Lemma \ref{lm:Laguerre_Hermite_Linfty_derivatives}.

\subsection{Lipschitz and $BMO$ properties}

Obviously, $\hfun\in L^\infty(\RR^d_+)$ for $\al\in[-1/2,\infty)^d$, hence $\hfun\in BMO(\RR^d_+)$. In order to justify that $\hfun\in\Lambda_{\nu}(\RR^d_+)$ for $\nu>0$ and certain $\al$'s, we shall consider the one-dimensional situation, and then apply Lemma \ref{lm:tensor_in_Lambda_nu}. To prove that $\hfunk\in\Lambda_{\nu}(\RR_+)$ we will construct an extension $\tilde{\varphi}_k^\al$ of $\hfunk$ to $\RR$, such that $\tilde{\varphi}_k^\al\in\Lambda_{\nu}(\RR)$. 

For $\al+1/2\notin\NN$ we simply put
\begin{equation*}
\tilde{\varphi}_k^\al(u)=\left\{
\begin{array}{ll}
\hfunk(u),&u>0,\\
0,&u\leq 0.
\end{array}\right.
\end{equation*} 
Observe that $\tilde{\varphi}_k^\al\in\mathcal{C}^{\lfloor \al+1/2\rfloor} (\RR)$
On the other hand, if $\al+1/2$ is an integer, then note that we can naturally extend the definition \eqref{eq:33} of $\funk$ to the whole $\RR$, and put
\begin{equation*}
\tilde{\varphi}_k^\al(u)=\hfunk(u),\qquad u\in\RR.
\end{equation*}  
In this case $\tilde{\varphi}_k^\al\in\mathcal{C}^{\infty} (\RR)$.

%\begin{equation*}
%\hfunk(x)=(-1)^{\al+1/2} \hfunk(|x|),\qquad x\in\RR.
%\end{equation*}
%Hence, by \eqref{eq:6} and \cite[(8) and (9)]{Plewa_2019_JFAA}, we have for $j\in\NN$
%\begin{equation*}
%\big\Vert (\hfunk)^{(j)}\big\Vert_{L^\infty (-1,1)}\lesssim (k+1)^{(2j-1)/4}. 
%\end{equation*}
%Finally, we define $\tilde{\varphi}_k^\al$ when $\al+1/2$ is an integer, by
%\begin{equation*}
%\tilde{\varphi}_k^\al(x)=\psi(x\sqrt{k+1})\hfunk(x),
%\end{equation*} 
%where $\psi$ is defined in \eqref{eq:psi_def}.

\begin{lm}\label{lm:Laguerre_Hermite_Lipschitz}
	Let $\al\geq -1/2$. If $\al+1/2\notin\NN$, then $\tilde{\varphi}_k^\al\in\Lambda_{\nu}(\RR)$ for $\nu \leq \al+1/2$, whereas if $\al+1/2\in\NN$, then $\tilde{\varphi}_k^\al\in\Lambda_{\nu}(\RR)$ for all $\nu\geq 0$.
\end{lm}

Notice that for $\al\in\{-1/2\}\cup[1/2,\infty)$ the functions $(\hfunk)'$ exist and are bounded, and observe that Lemmas \ref{lm:Laguerre_Hermite_Lipschitz} and \ref{lm:tensor_in_Lambda_nu} yield that for a given $p\in(0,1]$ and
\begin{equation*}
\al\in\Big(\Big\{-\frac{1}{2},\frac{1}{2},\ldots,P-\frac{1}{2}\Big\}\cup\Big[d\big(\frac{1}{p}-1\big)-\frac{1}{2},\infty\Big)\Big)^d,
\end{equation*}
where $P=\lfloor d(p^{-1}-1)\rfloor$, we have $\hfun\in\Lambda_{d(\frac{1}{p}-1)}(\RR^d_+)$.
 
For the proof of Lemma \ref{lm:Laguerre_Hermite_Lipschitz} we will need some auxiliary results.

\begin{lm}\label{lm:Laguerre_Hermite_fun_estim}
	Let $\al\geq-1/2$ and $j\in\NN$. Then, for any $c\in(0,1]$, we have 
	\begin{equation*}
	\big|(\hfunk)^{(j)}(u) \big|\lesssim\left\{
	\begin{array}{ll}
	u^{\al+1/2-j} (k+1)^{\al/2},& u\in\big(0,c(k+1)^{-1/2}\big),\\
	(k+1)^{j/2-1/4},&u\in\big(c(k+1)^{-1/2},1\big),
	\end{array}\right.
	\end{equation*}
	uniformly in $u$ and $k\in\NN$.
%	Consequently,
%	\begin{equation*}
%	\big|(\hfunk)^{(j)}(x) \big|\lesssim\left\{
%	\begin{array}{ll}
%	x^{\al+1/2-j} (k+1)^{\al/2},& \al\in[-1/2,j-1/2),\\
%	(k+1)^{j/2-1/4},&\al\in[j-1/2,\infty),
%	\end{array}\right.
%	\end{equation*}
%	uniformly in $x\in(0,1)$ and $k\in\NN$.
\end{lm}
\begin{proof}
	Fix $c\in(0,1]$. We will apply the induction over $j$. For $j=0$ the estimates are known (see \cite[(1)]{Plewa_2019_JFAA}, and for the original result \cite[p.~699]{Askey_Wainger_1965_AJM} and \cite[p.~435]{Muckenhoupt_1970}). We assume that the claim holds for $j\in\NN$ and will prove it for $j+1$. By \eqref{eq:6} we have
	\begin{equation*}
	(\hfunk)^{(j+1)}(u)= \frac{d^j}{du^j}\Big(-2\sqrt{k} \varphi_{k-1}^{\al+1}(u)+\Big(\frac{2\al+1}{2u}-u\Big)\hfunk(u)\Big).
	\end{equation*}
	Thus, $\big|(\hfunk)^{(j+1)}(u)\big|$ can be estimated from above by a constant multiple of
	\begin{equation*}
	\sqrt{k} \big| (\varphi_{k-1}^{\al+1})^{(j)}(u)\big|+\big| (\hfunk)^{(j-1)}(u) \big|+u\big| (\hfunk)^{(j)}(u) \big|+\sum_{\ell=0}^j u^{-\ell-1}\big| (\hfunk)^{(j-\ell)}(u) \big|,
	\end{equation*}
	where we set $(\hfunk)^{(-1)}\equiv 0 $.
	Finally, by the inductive hypothesis we obtain
	\begin{equation*}
	\big|(\hfunk)^{(j+1)}(u)\big|\lesssim   u^{\al+1/2-j} (k+1)^{\al/2}\Big(u(k+1)^{1/2}+u^{-1} +u  \Big) \lesssim u^{\al-1/2-j} (k+1)^{\al/2}.
	\end{equation*}
	uniformly in $u\in(0,c(k+1)^{-1/2})$, and
	\begin{equation*}
	\big|(\hfunk)^{(j+1)}(u)\big|\lesssim (k+1)^{(j+1)/2-1/4},
	\end{equation*}
	uniformly in $u\in (c(k+1)^{-1/2},1)$. This finishes the proof.
\end{proof}

The following result is an analogue of Lemma \ref{lm:standard_Laguerre_sharp_estimates}.

\begin{lm}\label{lm:Laguerre_Hermite_sharp_estimates}
	Let $\al\geq -1/2$ and $j,\ell\in\NN$ be given. There exists small constant $c>0$ such that there holds
	\begin{equation*}%\label{eq:??}
	\frac{d^j}{d u^j}\frac{\hfunk(u)}{u^{\al+1/2-\ell}}\simeq \left\{
	\begin{array}{ll}
	(k+1)^{\al/2} u^{\ell-j}, &\text{if}\quad \ell\geq j,\\
	(-1)^{\lceil \frac{j-\ell}{2}\rceil} (k+1)^{\al/2+\lceil \frac{j-\ell}{2}\rceil} u^{\frac{1-(-1)^{j-\ell}}{2}}, &\text{if}\quad \ell\leq j,
	\end{array}\right.
	\end{equation*}
	uniformly in $k\in\NN$ and $u\in(0,c(k+1)^{-1/2})$.
\end{lm}
\begin{proof}
	The proof is similar to the one of Lemma \ref{lm:standard_Laguerre_sharp_estimates}, therefore we will only sketch it. If $j=0$, then the estimate is well known (cf. \eqref{eq:22}). For $j\geq 1$ we use the induction and \eqref{eq:6}
	\begin{equation}\label{eq:23}
	\frac{d^{j+1}}{d u^{j+1}}\frac{\hfunk(u)}{u^{\al+1/2-\ell}}=\frac{d^j}{d u^j} \Big(\ell  \frac{\hfunk(u)}{u^{\al+1/2-\ell+1}} -  \frac{\hfunk(u)}{u^{\al-\ell-1/2}} -2\sqrt{k} \frac{\varphi_{k-1}^{\al+1}(u)}{u^{(\al+1)+1/2-\ell-1}}\Big).
	\end{equation}
	Note that if $\ell\geq j+1$, then the first component on the right hand side of the above identity is of the greatest size, $(k+1)^{\al/2+1/2} u^{\ell-j-1}$, and the others are strictly smaller.
	
	On the other hand, if $j\geq \ell$, then the second summand on the right hand side of \eqref{eq:23} is of the size (and sign)
	\begin{equation*}
	(-1)^{\lceil \frac{j-\ell-1}{2}\rceil+1} (k+1)^{\al/2+\lceil \frac{j-\ell-1}{2}\rceil} u^{\frac{1-(-1)^{j-\ell-1}}{2}},
	\end{equation*}
	and the third
	\begin{equation*}
	(-1)^{\lceil \frac{j-\ell-1}{2}\rceil+1} (k+1)^{\al/2+1+\lceil \frac{j-\ell-1}{2}\rceil} u^{\frac{1-(-1)^{j-\ell-1}}{2}}.
	\end{equation*}
	We see that the latter is the leading one. Moreover, by the simple identity $\lceil \frac{i-1}{2}\rceil+1=\lceil \frac{i+1}{2}\rceil$, $i\in\NN$, it can be written in the following form:
	\begin{equation*}
	(-1)^{\lceil \frac{j+1-\ell}{2}\rceil} (k+1)^{\al/2+\lceil \frac{j+1-\ell}{2}\rceil} u^{\frac{1-(-1)^{j+1-\ell}}{2}}.
	\end{equation*}
	Furthermore, the first component in \eqref{eq:23} can be decomposed similarly as in the proof of Lemma \ref{lm:Laguerre_Hermite_sharp_estimates}, and it gives the same growth and size as the remaining summands.
	
	This finishes the proof of the lemma.
\end{proof}

\begin{lm}\label{lm:Laguerre_Hermite_Linfty_derivatives}
	Let $j\in\NN$. For $\al\geq -1/2$ there holds
	\begin{equation}\label{eq:Laguerre_Hermite_Linfty_derrivatives_large_x}
	\left\Vert  (\hfunk)^{(j)}\right\Vert_{L^{\infty}(1/2,
		\infty)}\lesssim (k+1)^{(6j-1)/12},\qquad k\in\NN,
	\end{equation}
	whereas for $\al\in\{-1/2,1/2,\ldots,j-1/2\}\cup(j-1/2,\infty)$ there is
	\begin{equation}\label{eq:Laguerre_Hermite_Linfty_derrivatives}
	\left\Vert (\hfunk)^{(j)}\right\Vert_{L^{\infty}(\RR_+)}\lesssim (k+1)^{(6j-1)/12},\qquad k\in\NN.
	\end{equation}
\end{lm}
\begin{proof}
	In order to prove \eqref{eq:Laguerre_Hermite_Linfty_derrivatives_large_x} we use the induction over $j$ to prove an auxiliary result: for every $\ell\in\NN$ there is
	\begin{equation*}
	\sup_{u\geq 1/2} \big| u^\ell (\hfunk)^{(j)}(u)\big| \lesssim (k+1)^{(6(j+\ell)-1)/12}.
	\end{equation*}
	For $j=0$ we simply apply \eqref{eq:9}. Now assume that the claim holds for some $j\in\NN$. Observe that by \eqref{eq:6} we have for any $\ell\in\NN$
	\begin{align*}
	\big|u^\ell (\hfunk)^{(j+1)}(u)\big|&= u^\ell \Big|\frac{d^j}{du^j}\Big(-2\sqrt{k} \varphi_{k-1}^{\al+1}(u)+\Big(\frac{2\al+1}{2u}-u\Big)\hfunk(u)\Big)\Big|\\
	&\lesssim (k+1)^{(6(j+\ell)+5)/12} + \sum_{i=0}^j u^{\ell-1-j+i} \big| (\hfunk)^{(i)}(u)\big| +u^{\ell+1}\big| (\hfunk)^{(j)}(u)\big|\\
	&\lesssim (k+1)^{(6(j+\ell)+5)/12},
	\end{align*}
	uniformly in $k\in\NN$ and $u\geq 1/2$. This proves the auxiliary claim. Observe that for $\ell=0$ we obtain \eqref{eq:Laguerre_Hermite_Linfty_derrivatives_large_x}.
	
	To justify \eqref{eq:Laguerre_Hermite_Linfty_derrivatives} is suffices to verify that for the imposed $\al$ the required bound holds on the interval $(1,1/2)$. In fact, this is true even with the smaller exponent $(2j-1)/4$. Indeed, if $\al\geq j-1/2$, then we invoke Lemma \ref{lm:Laguerre_Hermite_fun_estim}, whereas in the case $j>\al+1/2\in\NN$ we additionally apply Lemma \ref{lm:Laguerre_Hermite_sharp_estimates} with $\ell=\al+1/2$. This finishes the proof of the lemma.
\end{proof}

\begin{lm}\label{lm:Laguerre_Hermite_fun_diff}
	For $j\in\NN$ and $\al\in (j-1/2,j+1/2]$ there holds
	\begin{equation*}
	\big| (\hfunk)^{(j)}(u)-(\hfunk)^{(j)}(u')\big| \lesssim (k+1)^{(2j+1)/4}|u-u'|+(k+1)^{\al/2}|u-u'|^{\al+1/2-j},
	\end{equation*}
	uniformly in $k\in\NN$ and $u,u'\in(0,1)$.
\end{lm}
\begin{proof}
	Fix $1>u>u'>0$. Observe that \eqref{eq:6} and Lemma \ref{lm:Laguerre_Hermite_fun_estim} permit to estimate 
	\begin{align*}
	\big|(\hfunk)^{(j)}(u)-(\hfunk)^{(j)}(u') \big|&=\Big|\int_{u'}^u \frac{d^j}{ds^{j}}\Big(-2\sqrt{k} \varphi_{k-1}^{\al+1}(s)+\Big(\frac{2\al+1}{2s}-s\Big)\hfunk(s) \Big)\,ds  \Big|\\
	&\lesssim \int_{u'}^u \Big(\sqrt{k} \big| (\varphi_{k-1}^{\al+1})^{(j)}(s)\big|+\sum_{\ell=0}^j s^{-\ell-1}\big| (\hfunk)^{(j-\ell)}(s) \big|\\
	&\qquad+\big| (\hfunk)^{(j-1)}(s) \big|+s\big| (\hfunk)^{(j)}(s) \big| \Big)\,ds\\
	&\lesssim |u-u'| (k+1)^{(2j+1)/4} + \sum_{\ell=0}^j \int_{u'}^u s^{-\ell-1}\big| (\hfunk)^{(j-\ell)}(s)\big|\,ds,
	\end{align*}		
	where we set $(\hfunk)^{(-1)}\equiv 0$.	Now notice that Lemma \ref{lm:Laguerre_Hermite_fun_estim} implies
	\begin{align*}
	&\int_{u'}^u s^{-\ell-1}\big| (\hfunk)^{(j-\ell)}(s)\big|\,ds\\
	&=\int_{[u',u]\cap[(k+1)^{-1/2},1)} s^{-\ell-1}\big| (\hfunk)^{(j-\ell)}(s)\big|\,ds+ \int_{[u',u]\cap(0,(k+1)^{-1/2})} s^{-\ell-1}\big| (\hfunk)^{(j-\ell)}(s)\big|\,ds\\
	&\lesssim |u-u'|(k+1)^{(2j+1)/4} + (k+1)^{\al/2}\int_{u'}^u s^{\al-1/2-j}\,ds.	
	\end{align*}		  
	Finally, since $\al\in(j-1/2,j+1/2]$ we see that
	\begin{equation*}
	\int_{u'}^u s^{\al-1/2-j}\,ds\lesssim |u-u'|^{\al+1/2-j}.
	\end{equation*}
	Combining the above gives the claim.
\end{proof}

\begin{proof}[Proof of Lemma \ref{lm:Laguerre_Hermite_Lipschitz}]
	We verify that the functions $\tilde{\varphi}_k^\al$ satisfy the condition in definition of $\Lambda_{\nu}(\RR)$.
	If $\al+1/2$ is an integer then the claim follows from \eqref{eq:Laguerre_Hermite_Linfty_derrivatives}. On the other hand, if $\al+1/2\notin \NN$, then we apply \eqref{eq:Laguerre_Hermite_Linfty_derrivatives}, \eqref{eq:Laguerre_Hermite_Linfty_derrivatives_large_x}, and Lemma \ref{lm:Laguerre_Hermite_fun_diff}.
%	
%	On the other hand, if $\al+1/2$ is an integer and $j\in\NN$, then
%	\begin{align*}
%	\Big\Vert\tilde{\varphi}_k^\al^{(j)}\Big\Vert_{L^\infty(\RR)}\lesssim (k+1)^{(6j-1)/12}. 
%	\end{align*}
%	Indeed, if $x>0$, then we apply \eqref{eq:Laguerre_Hermite_Linfty_derrivatives}, and for $x\in\big(-(k+1)^{-1/2},0\big)$ we have
%	\begin{align*}
%	\Big|\tilde{\varphi}_k^\al^{(j)}(x)\Big|=\Big|\sum_{\ell=0}^j c_\ell^j \psi^{(j-\ell)}(x\sqrt{k+1}) (k+1)^{(j-\ell)/2} (\hfunk)^{(\ell)}(x)\Big| &\lesssim \sum_{\ell=0}^j k^{j-\ell} (k+1)^{(6\ell-1)/12}\\
%	& \lesssim (k+1)^{(6j-1)/12},
%	\end{align*}
%	for some constants $c_\ell^j$. 
%	
%	
\end{proof}

\subsection{Hardy's inequality}
The kernels of the operators $\Rop$ (cf. \eqref{R_def}) associated with the Laguerre functions of Hermite type, are defined by 
\begin{equation}\label{eq:12}
\Rop(x,y)=\sum_{n\in\NN^d} r^{\ven} \hfun(x)\hfun(y) ,
\end{equation}
and, in the one-dimensional case, admit the explicit form (cf. \cite{Szego1959})
\begin{equation}\label{eq:2}
\Rop(u,v)=\frac{2 (uv)^{1/2}}{(1-r)r^{\al/2}}\exp\left(-\frac{1}{2}\frac{1+r}{1-r}(u^2+v^2)\right)I_{\al }\left(\frac{2r^{1/2}}{1-r}uv\right).
\end{equation}

Unfortunately, it is highly complicated to proceed as in \cite{Plewa_2019_JFAA} while estimating derivatives of $\Rop$ of order higher than $2$. The cancellations between the underlying Bessel functions are not well understood yet. Therefore, we choose an approach similar to the one applied in the case of the Jacobi expansions \cite{Plewa_2019_JAT}. This method relies on the following formula
\begin{equation}\label{eq:36}
I_\al(z)=z^{\al} \int_{-1}^1 e^{-zs} \Pi_\al(ds),\qquad |\arg z|<\pi,\ \al\geq -1/2,
\end{equation}
where $\Pi_\al$ in the case $\al>-1/2$ is a measure with the density given by
\begin{equation*}
\Pi_\al(ds)=\frac{(1-s^2)^{\al-1/2}ds}{\sqrt{\pi}\Gamma(\al+1/2)},
\end{equation*}
whereas for $\al=-1/2$ it is an atomic measure of the form $\Pi_{-1/2}=\frac{\delta_{-1}+\delta_{1}}{\sqrt{2\pi}}$.

Hence, by \eqref{eq:2} we have for $\al>-1/2$
\begin{equation*}
\Rop(u,v)= \frac{2^{\al+1} (uv)^{\al+1/2}}{(1-r)^{\al+1}} E_r^\al(u,v),
\end{equation*}
where by $E_r^{\al}(u,v)$ we denote
\begin{equation}\label{eq:40}
\exp\Big(-\frac{1}{2}\frac{1+r}{1-r}(v-u)^2-\frac{1-r}{(1+\sqrt{r})^2}uv\Big)\int_{-1}^1 \exp\Big(-\frac{2\sqrt{r}}{1-r}uv(s+1)  \Big) \frac{(1-s^2)^{\al-1/2}ds}{\sqrt{\pi}\Gamma(\al+1/2)}.
\end{equation}
Note that if $\al=-1/2$, then  
\begin{equation}\label{eq:4}
R_r^{-1/2}(u,v)=\frac{2}{\sqrt{\pi} \sqrt{1-r}}\exp\Big(-\frac{1}{2}\frac{1+r}{1-r}(u^2+v^2) \Big)\cosh\Big(\frac{2\sqrt{r}uv}{1-r} \Big).
\end{equation}

Now we have the following proposition.

\begin{prop}\label{prop:Laguerre_Hermite_R_deriv}
	For $j\in\NN$ and $\al\in\{-1/2,1/2,\ldots,j-1/2\}\cup(j-1/2,\infty)$ there holds
\begin{equation*}
\sup_{u>0} \Big\Vert \partial^j_u \Rop(u,\cdot)\Big\Vert_{L^2(\RR_+)}\lesssim (1-r)^{-\frac{1+2j}{4}},\qquad r\in(0,1).
\end{equation*}
\end{prop}

\begin{proof}
	Observe that Parseval's identity and \eqref{eq:Laguerre_Hermite_Linfty_derrivatives} yield
	\begin{equation*}
	\sup_{u>0} \Big\Vert \partial^j_u \Rop(u,\cdot)\Big\Vert_{L^2(\RR_+)}\leq \Big(\sum_{k=0}^\infty 2^{-2k} \Vert (\hfun)^{(j)}\Vert_{L^\infty(\RR_+)}^2  \Big)^{1/2}\lesssim 1, 
	\end{equation*}
	uniformly in $r\in(0,1/2)$. Hence, we can focus only on the case $r\in[1/2,1)$.
	
	We shall firstly consider the situation when $\al\geq 1/2$ and $j\in\NN_+$ (for $j=0$ see \cite[Lemma~3.1]{Plewa_2019_JFAA}). Note that for $\ell\in\NN$ such that $\ell\leq j$ we can write $\partial^{\ell}_u E_r^\al(u,v)$, where $E^\al_r(u,v)$ is defined in \eqref{eq:40}, as
\begin{align*}
&\qquad\int_{-1}^1 \exp\Big(-\frac{1}{2}\frac{1+r}{1-r}(v-u)^2-\frac{1-r}{(1+\sqrt{r})^2}uv-\frac{2\sqrt{r}}{1-r}uv(s+1)  \Big)\\
&\times\sum_{\stackrel{k,i\geq 0}{k+2i=\ell}}c_{k,i}^\ell (1-r)^{-k}\Big((1+r)(u-v)+\frac{(1-r)^2}{(1+\sqrt{r})^2}y+2\sqrt{r}uv(s+1) \Big)^k \Big(\frac{1+r}{1-r}\Big)^i     \Pi_\al (ds),
\end{align*}
where $c_{k,i}^l$ are certain constants (cf. \cite[p.~812]{Nowak_Szarek_2012_JMAA}). Hence,
\begin{align*}
\Big| \partial^{\ell}_u E_r^\al(u,v)\Big|
&\lesssim \exp\Big(-\frac{1}{2}\frac{(v-u)^2}{1-r}\Big)(1-r)^{-\ell/2}\bigg(1+\min\Big(\frac{\sqrt{1-r}}{u},(1-r)^{3/2}v\Big)\bigg)^\ell\\
&\qquad\times  \int_{-1}^1 \exp\Big(-\frac{\sqrt{r}}{1-r}uv(s+1)  \Big)(1-s^2)^{\al-1/2}\,ds\\
&\lesssim (1-r)^{-\ell/2} \exp\Big(-\frac{1}{2}\frac{(v-u)^2}{1-r}\Big)\bigg(1+\min\Big(\frac{\sqrt{1-r}}{u},(1-r)^{3/2}v\Big)\bigg)^\ell\\
&\qquad\times  \int_{-1}^1 \exp\Big(-\frac{\sqrt{r}}{1-r}uv(s+1)  \Big)(1+s)^{\al-1/2}\,ds,
\end{align*}
uniformly in $u,v>0$ and $r\in(1/2,1)$. The latter integral is estimated by a constant. On the other hand, again uniformly in $u,v>0$ and $r\in(1/2,1)$,
\begin{align*}
\int_{-1}^1 \exp\Big(-\frac{\sqrt{r}}{1-r}uv(s+1)  \Big)(1+s)^{\al-1/2}\,ds&\lesssim \Big(\frac{1-r}{uv}\Big)^{\al-1/2}\int_{0}^\infty \exp\Big(-\frac{\sqrt{r}}{1-r}uvs  \Big)\,ds\\
&\simeq \Big(\frac{1-r}{uv}\Big)^{\al+1/2}.
\end{align*}

Now we are ready to estimate $\partial^j_u\Rop(u,v)$. Combining the above we obtain
\begin{align*}
\big|\partial^j_u\Rop(u,v)\big|
&\leq \frac{2^{\al+1} v^{\al+1/2}}{(1-r)^{\al+1}}\sum_{\ell} {j\choose \ell} \big|\partial_u^\ell E_r^\al(u,v) \big| \big|\partial_u^{j-\ell}u^{\al+1/2}\big|\\
&\lesssim \Big(\frac{uv}{1-r}\Big)^{\al+1/2}\sum_{\ell} \Big(\frac{\sqrt{1-r}}{u}\Big)^{j-\ell} (1-r)^{-(j+1)/2}\exp\Big(-\frac{1}{2}\frac{(v-u)^2}{1-r}\Big)\\
&\quad\times
 \bigg(1+\min\Big(\frac{\sqrt{1-r}}{u},(1-r)^{3/2}v\Big)\bigg)^\ell \min\Big(1, \Big(\frac{1-r}{uv}\Big)^{\al+1/2}\Big)\\
 &\lesssim (1-r)^{-(j+1)/2} \Big(\frac{\sqrt{1-r}}{u}\Big)^{j}  \min\Big(\frac{uv}{1-r},1\Big)^{\al+1/2}\exp\Big(-\frac{1}{2}\frac{(v-u)^2}{1-r}\Big)\\
 &\quad\times\max_{\ell}\Big(\frac{u}{\sqrt{1-r}}+\min\big(1,(1-r)uv\big)\Big)^\ell,
\end{align*}
where $\ell\in\{0,\ldots,j\}$ if $j\leq \al+1/2$, and $\ell\in\{j-(\al+1/2),\ldots,j\}$ if $\al+1/2$ is an integer and $j>\al+1/2$. Observe that if $u\geq \sqrt{1-r}$, then
\begin{equation*}
\big|\partial^j_u\Rop(u,v)\big|\lesssim  (1-r)^{-(j+1)/2}\exp\Big(-\frac{1}{2}\frac{(v-u)^2}{1-r}\Big).
\end{equation*}
In the other case, $u\leq \sqrt{1-r}$ we estimate firstly assuming that $j\leq\al+1/2$
\begin{align*}
\big|\partial^j_u\Rop(u,v)\big|&\lesssim (1-r)^{-(j+1)/2} \Big(\frac{\sqrt{1-r}}{u}\Big)^{j} \Big(\frac{uv}{1-r}\Big)^{j} \exp\Big(-\frac{1}{2}\frac{(v-u)^2}{1-r}\Big)\\
&\lesssim (1-r)^{-(j+1)/2} \Big(\frac{v}{\sqrt{1-r}}\Big)^{j} \exp\Big(-\frac{1}{2}\frac{(v-u)^2}{1-r}\Big)\\
&\lesssim  (1-r)^{-(j+1)/2} \Big(\frac{|v-u| +u}{\sqrt{1-r}}\Big)^{j} \exp\Big(-\frac{1}{2}\frac{(v-u)^2}{1-r}\Big)\\
&\lesssim (1-r)^{-(j+1)/2} \exp\Big(-\frac{1}{4}\frac{(v-u)^2}{1-r}\Big).
\end{align*}
If $\al+1/2\in\NN$ is smaller than $j$, then we bound $\big|\partial^j_u\Rop(u,v)\big|$ by a constant multiple of $(1-r)^{-(j+1)/2}$ times
\begin{align*}
&\hspace{-0.5cm}\Big(\frac{\sqrt{1-r}}{u}\Big)^{j} \Big(\frac{uv}{1-r}\Big)^{\al+1/2} \Big(\frac{u}{\sqrt{1-r}}+(1-r)uv\Big)^{j-\al-1/2}\exp\Big(-\frac{1}{2}\frac{(v-u)^2}{1-r}\Big)\\
&\lesssim  \big(1+(1-r)^{3/2}y\big)^{j}\Big(\frac{\frac{v}{\sqrt{1-r}}}{1+(1-r)^{3/2}v}\Big)^{\al+1/2}\exp\Big(-\frac{1}{2}\frac{(v-u)^2}{1-r}\Big)\\
&\lesssim\big(1+(1-r)^{3/2}\big((v-u)+u\big)\big)^{j-\al-1/2} \Big(\frac{(v-u)+u}{\sqrt{1-r}}\Big)^{\al+1/2} \exp\Big(-\frac{1}{2}\frac{(v-u)^2}{1-r}\Big)\\
&\lesssim \exp\Big(-\frac{1}{4}\frac{(v-u)^2}{1-r}\Big).
\end{align*}

Combining the above we arrive at
\begin{align}\label{eq:5}
\begin{split}
\sup_{u>0} \big\Vert \partial_u^j \Rop(u,\cdot)\big\Vert_{L^2(\RR_+)}&\lesssim (1-r)^{-(j+1)/2} \sup_{u>0} \Big( \int_{\RR_+}  \exp\Big(-\frac{1}{4}\frac{(v-u)^2}{1-r}\Big)\,dv\Big)^{1/2}\\
&= 2\pi^{1/4} (1-r)^{-(2j+1)/4},
\end{split}
\end{align}
and this completes the proof of the proposition for $\al\geq 1/2$.

Now we move on to the case $\al<1/2$. In fact, we need to consider only $\al=-1/2$ and $j\in\NN$, since for $\al\in(-1/2,1/2)$ only $j=0$ is allowed, and this was already done in author's previous paper (see \cite[Lemma~3.1]{Plewa_2019_JFAA}). By \eqref{eq:4} we obtain for some constants $c_{k,i}^j$ and $\tilde{c}_{k,i}^j$ the following equality
\begin{align*}
\begin{split}
\partial_u^j R_r^{-1/2}(u,v)
&=\frac{1}{\sqrt{\pi} \sqrt{1-r}}\bigg(\exp\Big(-\frac{1}{2}\frac{1+r}{1-r}(v-u)^2-\frac{1-r}{(1+\sqrt{r})^2}uv -\frac{4\sqrt{r}}{1-r}uv\Big)\\
&\quad\times \sum c_{k,i}^j \Big(\frac{1+r}{1-r}(u-v)+\frac{1-r}{(1+\sqrt{r})^2}v+\frac{4\sqrt{r}}{1-r}v\Big)^k\Big(\frac{1+r}{1-r}\Big)^i \\
&\quad +\exp\Big(-\frac{1}{2}\frac{1+r}{1-r}(v-u)^2-\frac{1-r}{(1+\sqrt{r})^2}uv\Big)\\
&\quad\times\sum\tilde{c}_{k,i}^j \Big(\frac{1+r}{1-r}(u-v)+\frac{1-r}{(1+\sqrt{r})^2}v\Big)^k\Big(\frac{1+r}{1-r}\Big)^i \bigg),
\end{split}
\end{align*}
where in both sums the summation goes over all $k,i\geq 0$ such that $k+2i=j$. 
Hence, 
\begin{align*}
\big|\partial_u^j R_r^{-1/2}(u,v)\big|&\lesssim (1-r)^{-(j+1)/2} \Big(1+\min\big( \frac{v}{\sqrt{1-r}},\frac{\sqrt{1-r}}{u}\big)\Big)^j \exp\Big(-\frac{1}{2}\frac{1+r}{1-r}(v-u)^2\Big)\\
&\lesssim (1-r)^{-(j+1)/2}  \Big(\frac{v-u}{\sqrt{1-r}}+1\Big)^j \exp\Big(-\frac{1}{2}\frac{1+r}{1-r}(v-u)^2\Big)\\
&\lesssim (1-r)^{-(j+1)/2} \exp\Big(-\frac{1}{2}\frac{1+r}{1-r}(v-u)^2\Big),
\end{align*}
where in the last but one inequality we used the simple estimate 
\begin{equation*}
\min(a+b,b^{-1})\leq a+1,\qquad a,b>0.
\end{equation*}
The last step is the same as in \eqref{eq:5}. This concludes the proof of the proposition.
\end{proof}

Before we state Hardy's inequality associated with the Laguerre functions of Hermite type we will prove some auxiliary results. The next one complements the estimate from Proposition \ref{prop:Laguerre_Hermite_R_deriv}. Essentially, it says that the mentioned bound holds also for $\al\in(j-3/2,j-1/2)$, $j\in\NN_+$, but only away from the origin. 

\begin{lm}\label{lm:Laguerre_Hermite_R_deriv_large_x}
	If $j\in\NN$ and $\al\in(j-1/2,j+1/2)$, then
	\begin{equation*}
	\sup_{u\geq 1/2} \Big\Vert \partial^{j+1}_u \Rop(u,\cdot)\Big\Vert_{L^2(\RR_+)}\lesssim (1-r)^{-\frac{3+2j}{4}},\qquad r\in(0,1).
	\end{equation*}
\end{lm}
\begin{proof}
	It suffices to proceed as in the proof of Proposition \ref{prop:Laguerre_Hermite_R_deriv} with some minor changes. For $r\in(0,1/2]$ use \eqref{eq:Laguerre_Hermite_Linfty_derrivatives_large_x} instead of  \eqref{eq:Laguerre_Hermite_Linfty_derrivatives}. If $r\in(1/2,1)$, then we arrive at
	\begin{align*}
	\big|\partial^{j+1}_u\Rop(u,v)\big|
	&\lesssim \sum_{\ell=0}^{j+1} \Big(\frac{\sqrt{1-r}}{u}\Big)^{j+1-\ell} (1-r)^{-(j+2)/2}\exp\Big(-\frac{1}{2}\frac{(v-u)^2}{1-r}\Big)	\Big(1+\frac{\sqrt{1-r}}{u}\Big)^\ell\\
	&\qquad\times 
 \min\Big(\frac{uv}{1-r},1\Big)^{\al+1/2}\\
	&\lesssim (1-r)^{-(j+2)/2} \sum_{\ell=0}^{j+1}\Big(\frac{\sqrt{1-r}}{u}\Big)^{j+1-\ell} \Big(1+\frac{\sqrt{1-r}}{u}\Big)^\ell \exp\Big(-\frac{1}{2}\frac{(v-u)^2}{1-r}\Big)\\
	&\lesssim (1-r)^{-(j+2)/2} \exp\Big(-\frac{1}{2}\frac{(v-u)^2}{1-r}\Big),
	\end{align*}
	since $\sqrt{1-r}\lesssim u$. Then we estimate like in \eqref{eq:5}. This finishes the proof of the lemma.
\end{proof}

%In order to prove this lemma we will need an auxiliary estimate on the derivatives of the functions $\hfunk(x)$ for small $x$.

Notice that Lemmas \ref{lm:Laguerre_Hermite_R_deriv_large_x} and \ref{lm:Laguerre_Hermite_fun_diff} yield for $j\in\NN$ and $\al\in(j-1/2,j+1/2)$ the estimate
\begin{align}\label{eq:14} 
\big\Vert \partial^j_u \Rop(u,\cdot)-\partial_u^j\Rop(u',\cdot)\big\Vert_{L^2(\RR_+)}\lesssim (1-r)^{-\frac{1+2j}{4}}|u-u'| +(1-r)^{\frac{\al+1}{2}} |u-u'|^{\frac{2\al+1-2j}{2}},
\end{align}
uniformly in $r\in(0,1)$ and $u,u'>0$ such that $|u-u'|\leq 1/2$. Indeed, if $u,u'\in(0,1)$ then we use Lemma \ref{lm:Laguerre_Hermite_fun_diff} and Parseval's identity. In the opposite case, $u,v\geq 1/2$, invoke the mean value theorem and Lemma \ref{lm:Laguerre_Hermite_R_deriv_large_x}.

\begin{prop}\label{prop:Laguerre_Hermite_cond_C}
	If $k\in\NN$ and $\al\in\big(\{-1/2,1/2,\ldots,k-1/2\}\cup (k-1/2,\infty)\big)^d$, then
	\begin{align*}
	\Big\Vert &\Rop(x,\cdot)-\sum_{\ven\leq k} \frac{\partial^{n}_{x} \Rop(x',\cdot)}{n_1! \cdot\ldots\cdot n_d!}\prod_{i=1}^d (x_i-x_i')^{n_i}\Big\Vert_{L^2(\RR^d_+)}\lesssim \sum_{\delta\in\Delta^\al_k}(1-r)^{-\frac{d+2k+2\delta}{4}}|x-x'|^{k+\delta},
	\end{align*}		
	uniformly in $r\in(0,1)$ and $x,x'\in\RR^d_+$ such that $|x-x'|\leq 1/2$, where
	\begin{equation}\label{eq:7} 
	\Delta^\al_k=\{1\}\cup\{\al_i+1/2-k:\ \al_i\in(k-1/2,k+1/2) \}.
	\end{equation}
\end{prop}
\begin{proof}
	The proof is analogous to the one of Proposition \ref{prop:standard_Laguerre_cond_C}, thus we will only sketch it. 
	
	Observe that if $\al_i\notin(k-1/2,k+1/2)$ for all $i=1,\ldots,d$, then the claim, with $\Delta^\al_k=\{1\}$, follows from Taylor's theorem and Proposition \ref{prop:Laguerre_Hermite_R_deriv} applied for $j=k+1$ . On the other hand, if $\al_i\in(k-1/2,k+1/2)$ for some $i$ then we apply Taylor's theorem, Proposition \ref{prop:Laguerre_Hermite_R_deriv}, and \eqref{eq:14}. Then the set $\Delta^\al_k$ is as in \eqref{eq:7}. We omit the details.	
\end{proof}

Now we are ready to state Hardy's inequality associated with the system of Laguerre functions of Hermite type.

\begin{thm}\label{thm:main_Laguerre_Hermite}
	Let $p\in(0,1)$, $s\in[p,2]$, and denote $P:=\lfloor d(\frac{1}{p}-1)\rfloor$. For
	\begin{equation*}
	\al\in\big(\{-1/2,1/2,\ldots,P-1/2\}\cup (d(p^{-1}-1)-1/2,\infty)\big)^d,
	\end{equation*}
	there holds
	\begin{equation*}
	\sum_{n\in\NN^d}\frac{|\langle f,\hfun\rangle|^s}{(\ven+1)^E}\lesssim \Vert f\Vert^s_{H^p(\RR^d_+)}, \qquad f\in H^p(\RR^d_+),
	\end{equation*}
	where $E= d+\frac{ds}{4p}(2-3p)$, and the exponent is sharp.
\end{thm}
\begin{proof}
	Similarly as in Theorem \ref{thm:main_Laguerre_standard}: the inequality follows from Theorem \ref{thm:Hp_general} and Proposition \ref{prop:Laguerre_Hermite_cond_C}, whereas sharpness is a consequence of Propositions \ref{prop:general_sharpness}, \ref{prop:general_sharpness_2} and Lemma \ref{lm:Laguerre_Hermite_sharp_estimates}. The case $\al_i+1/2=d(p^{-1}-1)$ is excluded due to Remark \ref{rem:sharpness}, unless it is an integer.
\end{proof}

\subsection{Heat kernel estimates}\label{subS:heat_kern_estim}

In this article we estimated or will estimate the kernels $R_r(x,y)$ in various contexts. In case of the standard Laguerre functions it was very easy and for the Jacobi expansions we will use the result known in the literature. On the other hand, here the situation was more involved. In Proposition \ref{prop:Laguerre_Hermite_R_deriv} we have obtained a result which can be interesting on its own, especially in the context of the associated heat kernel.

Recall that the heat semigroup $\{T^\al_t\}_{t\geq 0}$ is spectrally defined by
\begin{equation*}
T^\al_t f =\sum_{n\in\NN^d} e^{-t(4\ven+2|\al|+2d)} \langle f,\hfun\rangle \hfun,\qquad f\in L^2(\RR^d_+).
\end{equation*}

It is known (cf. \cite[p.~403]{Nowak&Stempak_2007_JFA}) that $T_t$ are integral operators:
\begin{equation*}
T_t^\al f(x) =\int_{\RR^d_+} G^\al_t(x,y)f(y)\,dy,\qquad f\in L^2(\RR^d_+),\ x\in\RR^d_+,
\end{equation*}
where
\begin{equation*}
G_t^\al (x,y) =\sum_{n\in\NN^d} e^{-t(4\ven+2|\al|+2d)} \hfun(x) \hfun(y),
\end{equation*}
and explicitly (cf. \cite[(4.17.6)]{Lebedev1972})
\begin{equation*}
G_t^\al (x,y)=(\sinh 2t)^{-d} \exp \Big(-\frac{1}{2}\coth (2t) (|x|^2+|y|^2)\Big) \prod_{i=1}^d \sqrt{x_i y_i} I_{\al_i} \Big(\frac{x_i y_i}{\sinh 2t}\Big).
\end{equation*}

Observe that by the definition of $G_t^\al$ and \eqref{eq:12} we have the following relation
\begin{equation*}
G_t^\al (x,y)=e^{-2t(|\al|+d)} R^\al_{e^{-4t}}(x,y).
\end{equation*}
Hence, the results obtained for $\Rop(x,y)$ can be easily transferred to $G^\al_t(x,y)$. Therefore, by \eqref{eq:5} we have the following one-dimensional estimate. By an obvious modification, this lemma can be generalized to $d\geq 1$.

\begin{lm}
	If $j\in\NN$ and $\al\in\{-1/2,1/2,\ldots,j-1/2\}\cup(j-1/2,\infty)$, then
	\begin{equation*}
	\big\vert \partial^j_u G_t^\al(u,v)\big\vert\lesssim
	\left\{ \begin{array}{cc}
	t^{-\frac{j+1}{2}} \exp \Big(-c\frac{(u-v)^2}{t} \Big) ,& t\leq 1,\\
	e^{-2t(\al+1)} e^{-c(u-v)^2},& t\geq 1,
	\end{array}\right.
	\end{equation*}
	uniformly in $u,v,t>0$ and for some positive constant $c$. Moreover,
	\begin{equation*}
	\sup_{u>0} \big\Vert \partial^j_u G_t^\al(u,\cdot)\big\Vert_{L^2(\RR_+)}\lesssim
	\left\{ \begin{array}{cc}
	t^{-\frac{2j+1}{4}},& t\leq 1,\\
	e^{-2t(\al+1)},& t\geq 1.
	\end{array}\right.
	\end{equation*}
\end{lm} 

\subsection{Generalized Hermite functions}
In this subsection we focus on the generalized Hermite function system. Due to its relation with the Laguerre expansions of Hermite type, we will essentially deduce the desired results from the analogous ones above.   

The {\it generalized Hermite functions} $h_k^\lambda$, $k\in\NN$, of order $\lambda\geq 0$ on $\RR$ are defined via
\begin{equation*}
h_{2k}^\lambda(u)= (-1)^k 2^{-1/2}\varphi_k^{\lambda-1/2}(|u|),\quad h_{2k+1}^\lambda(u)=(-1)^k 2^{-1/2}\mathrm{sgn}(u)\varphi_k^{\lambda+1/2}(|u|),\quad u\in\RR,
\end{equation*}
where for $u=0$ we naturally extend the definition of $\hfunk$ from \eqref{eq:33}. In higher dimensions these functions are defined as tensor products, similarly as in the previous sections. The system $\{\hen\}_{n\in\NN^d}$ forms an orthonormal basis in $L^2(\RR^d)$. We remark that $\{h_n^{(0,\ldots,0)}\}_{n\in\NN^d}$ is the Hermite function basis. 

The generalized Hermite functions $\{h_n^\lambda\}_{n\in\NN^d}$ are bounded (cf. \eqref{eq:9}), and therefore they are in $BMO(\RR^d)$. Moreover, for $\lambda\in \{0,2,\ldots\}\cup[p^{-1}-1,\infty)$ they belong to the Lipschitz spaces $\Lambda_{ \frac{1}{p}-1 }$, see \cite[Proposition~1.2]{LiShi_2014_CA}. Hence, by Lemma \ref{lm:tensor_in_Lambda_nu} in the multi-dimensional situation we see that $h_n^\lambda\in\Lambda_{d (\frac{1}{p}-1)}$ for $\lambda\in \big(\{0,2,\ldots\}\cup [d(p^{-1}-1),\infty)\big)^d$ (note that the additional assumption is satisfied).

The family of kernels $R_r(u,v)$ associated with the generalized Hermite functions, in the case $d=1$, is given by
\begin{equation*}
\htRop(u,v)=\sum_{k\in\NN} r^k h^\lambda_k(u) h_k^\lambda(v). 
\end{equation*}
We use the symbol $\tilde{R}$ instead of $R$ to distinct this kernel from the one associated with the functions $\{\hfunk\}_{k\in\NN}$. Notice that
\begin{align}\label{eq:37}
\htRop(u,v) =\frac{1}{2}\Big( R^{\lambda-1/2}_{r^2}(|u|,|v|) +{\rm sgn}(uv)r R_{r^2}^{\lambda+1/2} (|u|,|v|)\Big),
\end{align}
where $r\in(0,1)$ and $u,v\in\RR$. Here $R$ denotes the kernel corresponding to the Laguerre functions of Hermite type. We naturally extended the definition for $u=0$ and $v=0$.  Observe that if $\lambda$ is an even integer, then $\htRop\in \mathcal{C}^\infty (\RR\times \RR)$. Moreover, given $j\in\NN$ we see that $\htRop\in \mathcal{C}^j (\RR\times \RR)$ for $\lambda> j$.

Fix $j\in\NN$ and $\lambda\in \{0,2,\ldots\}\cup(j,\infty)$. For $u\neq0$ we have
\begin{equation*}
\partial^j_u \htRop(u,v)=\frac{({\rm sgn} u)^j}{2}\Big( \partial^j_u R^{\lambda-1/2}_{r^2}(|u|,|v|) +{\rm sgn}(vu)r \partial_u^j  R_{r^2}^{\lambda+1/2} (|u|,|v|)\Big), 
\end{equation*}
whereas for $u=0$ we see that
%\begin{equation*}
%\partial^j_u \htRop(u,v)=\frac{(-1)^j}{2}\Big( \partial^j_u  R^{\lambda-1/2}_{r^2}(-u,|v|) -{\rm sgn}(v)r \partial_u^j  R_{r^2}^{\lambda+1/2} (-u,|v|)\Big),
%\end{equation*}
%and finally, for $u=0$ we see that
\begin{equation*}
\partial^j_u \htRop(0,v)=\frac{1}{2}\partial^j_u R^{\lambda-1/2}_{r^2}(0,|v|) , 
\end{equation*}
where in both cases $r\in(0,1)$ and $v\in\RR$. In the latter equality we naturally extended the formula from \eqref{eq:2} to $u=0$.

\begin{lm}
	Let $j\in\NN$. For $\lambda\in\{0,2,\ldots \}\cup (j,\infty)$ we have
	\begin{equation*}
	\sup_{u\in\RR} \big\Vert \partial_u^j \htRop(u,\cdot)\big\Vert_{L^2(\RR)}\lesssim (1-r)^{(1+2j)/4},
	\end{equation*} 
	uniformly in $r\in(0,1)$. Moreover, for $\lambda\in (j,j+1]$ we have
	\begin{equation*}
	\big\Vert \partial_u^j \htRop(u,\cdot)-\partial_u^j \htRop(u',\cdot)\big\Vert_{L^2(\RR)}\lesssim (1-r)^{(1+2j)/4} |u-u'|+ (1-r)^{(2\lambda+1)/4}|u-u'|^{\lambda-j}
	\end{equation*}
	uniformly in $r\in(0,1)$ and $u,u'\in\RR$ such that $|u-u'|\leq 1$.
\end{lm}
\begin{proof}
	The first part follows from \eqref{eq:37} and Proposition \ref{prop:Laguerre_Hermite_R_deriv}. For the second see \eqref{eq:14}.
\end{proof}

Now the version of \eqref{cond:C} corresponding to the generalized Hermite setting follows easily. Then we immediately obtain associated Hardy's inequality.

\begin{prop}\label{prop:generalized_Hermite_cond_C}
	If $k\in\NN$ and $\lambda\in\big(\{0,2,\ldots\}\cup (k,\infty)\big)^d$, then
	\begin{align*}
	\Big\Vert &\htRop(x,\cdot)-\sum_{\ven\leq k} \frac{\partial^{n}_{x} \htRop(x',\cdot)}{n_1! \cdot\ldots\cdot n_d!}\prod_{i=1}^d (x_i-x_i')^{n_i}\Big\Vert_{L^2(\RR^d)}\lesssim \sum_{\delta\in\Delta^\al_k}(1-r)^{-\frac{d+2k+2\delta}{4}}|x-x'|^{k+\delta},
	\end{align*}		
	uniformly in $r\in(0,1)$ and $x,x'\in\RR^d$ such that $|x-x'|\leq 1/2$, where
	\begin{equation*}%\label{eq:7} 
	\Delta^\lambda_k=\{1\}\cup\{\lambda_i-k:\ \lambda_i\in(k,k+1) \}.
	\end{equation*}
\end{prop}

\begin{thm}\label{thm:main_generalized_Hermite}
	Let $p\in(0,1)$, $s\in[p,2]$, $d\in\NN$, and $P=\lfloor d(p^{-1}-1)\rfloor$. For $\lambda\in\big(\{0,2,\ldots\}\cup (d(p^{-1}-1),\infty)\big)^d$, there holds
	\begin{equation*}
	\sum_{n\in\NN^d}\frac{|\langle f,\hen\rangle|^s}{(\ven+1)^E}\lesssim \Vert f\Vert^s_{H^p(\RR^d)}, \qquad f\in H^p(\RR^d),
	\end{equation*}
	where $E= d+\frac{ds}{4p}(2-3p)$, and the exponent is sharp.
\end{thm}
\begin{proof}
	The inequality is a consequence of Proposition \ref{prop:generalized_Hermite_cond_C} and Theorem \ref{thm:Hp_general}. On the other hand, sharpness follows immediately from Theorem \ref{thm:main_Laguerre_Hermite}. Indeed, it is clear that if the admissible exponent for the generalized Hermite functions could be lowered, so could be the exponent corresponding to the Laguerre expansions of Hermite type. 
\end{proof}

We remark that for $\lambda=(0,\ldots,0)$, that is in the case of the Hermite functions, the result agrees with the ones already known in the literature (\cite{Radha_Thangavelu_2004_PAMS} for $d\geq 2$, \cite{LiYuShi_2015_JFAA} for $d=1$).

\section{Jacobi trigonometric functions}\label{S:Jacobi_trigonom_fun}
For the type parameters $\al,\be>-1$ the {\it Jacobi functions} $\jfunk$, $k\in\NN$, are defined by
\begin{equation}\label{eq:jfunk_definition}
\jfunk(\te)=\Big(\sin\frac{\te}{2}\Big)^{\al+1/2}\Big(\cos\frac{\te}{2}\Big)^{\be+1/2} \jtpolk(\te),\qquad \te\in(0,\pi),
\end{equation}
where 
\begin{equation*}
\jtpolk(\te)= c_k^{\al,\be} \jpolk(\cos \te),
\end{equation*}
and $\jpolk$ denotes the Jacobi polynomial of type $\al,\be$ and degree $k$.
Here $c_k^{\al,\be}$ is the normalizing constant,
\begin{align*}
c_k^{\al,\be}=\bigg(\frac{(2k+\al+\be+1)\Gamma(k+\al+\be+1)\Gamma(k+1)}{\Gamma(k+\al+1)\Gamma(k+\be+1)} \bigg)^{1/2},
\end{align*}
where for $k=0$ and $\al+\be = -1$ we write $1$ in place of $(2k+\al+\be + 1)\Gamma(k+\al+\be + 1)$ in the numerator. Note that $c_k^{\al,\be}\simeq (k+1)^{1/2}$, $k\in\NN$. The system $\{\jfunk\}_{k\in\NN}$ is an orthonormal basis in $L^2((0,\pi))$. In higher dimensions $\jfun(\te)$ are defined as tensor products of $\jfuni(\te_i)$.

We are now interested in the $L^\infty$ norms of the derivatives of $\jfunk$ in various ranges of the parameters $\al$ and $\be$ and on different subintervals of $(0,\pi)$. Firstly, recall that for $\al,\be\geq -1/2$ there is (see \cite[(2.8)]{Muckenhoupt_1986_MAMS})
\begin{equation}\label{eq:27}
\big|\jfunk(\te) \big|\lesssim \left\{ \begin{array}{ll}
\big((k+1)\te\big)^{\al+1/2},& 0<\te\leq (k+1)^{-1},\\
1, & (k+1)^{-1} \leq \te\leq \pi-(k+1)^{-1},\\
\big((k+1)\te\big)^{\be+1/2},& \pi-(k+1)^{-1}\leq \te< \pi.
\end{array}
\right.
\end{equation}
Hence, for $\al,\be\geq -1/2$
\begin{equation*}
\Vert\jfunk\Vert_{L^\infty(0,\pi)}\lesssim 1,\qquad k\in\NN.
\end{equation*}
Secondly, we make use of the formula (cf. \cite[(4.21.7)]{Szego1959} or \cite[p.~364]{Ciaurri&Nowak&Stempak_2007_MZ} after an obvious simplification)
\begin{align}\label{eq:39}
\begin{split}
\frac{d}{d\te} \jfunk(\te)= - k_{\al,\be} \phi_{k-1}^{\al+1,\be+1}(\te) \Big(\frac{2\al+1}{4}\cot\frac{\te}{2}-\frac{2\be+1}{4}\tan\frac{\te}{2}\Big)\jfunk(\te),
\end{split}
\end{align}
where we put $\phi_{-1}^{\al+1,\be+1}\equiv 0$ and $k_{\al,\be}=\sqrt{k(k+\al+\be+1)} $. Observe that \eqref{eq:39} and \eqref{eq:27} give for $\al\in\{-1/2\}\cup[1/2,\infty)$ and $\be\geq -1/2 $ the bound
\begin{align*}
\big\Vert (\jfunk)'\big\Vert_{L^\infty(0,\frac{2\pi}{3})}\lesssim (k+1),\qquad k\in\NN,
\end{align*}
whereas for $\al\geq -1/2$ and  $\be\in\{-1/2\}\cup[1/2,\infty)$,
\begin{align*}
\big\Vert (\jfunk)'\big\Vert_{L^\infty(\frac{\pi}{3},\pi)}\lesssim (k+1),\qquad k\in\NN.
\end{align*}
For similar estimates for higher order derivatives see Lemma \ref{lm:Jacobi_fun_Linfty_derivatives}.

We will frequently make use of the formula
\begin{equation}\label{eq:30}
\jfunk(\te)=\phi_k^{\be,\al}(\pi-\te),\qquad \te\in(0,\pi).
\end{equation}

\subsection{Lipschitz and $BMO$ properties}

Let us firstly give some auxiliary lemmas and justify that the Jacobi functions belong to the Lipschitz spaces $\Lambda_{\nu}((0,\pi))$ for certain $\nu$.

\begin{lm}\label{lm:Jacobi_fun_estim}
	Let $j\in\NN$ and $\al,\be\geq-1/2$. Then, for any $c\in(0,1]$, we have
	\begin{equation*}
	\big|(\jfunk)^{(j)}(\te) \big|\lesssim\left\{
	\begin{array}{ll}
	\te^{\al+1/2-j} (k+1)^{\al+1/2},& \te\in\big(0,c(k+1)^{-1}\big),\\
	(k+1)^{j},&\te\in\big(c(k+1)^{-1},\frac{2\pi}{3}\big),
	\end{array}\right.
	\end{equation*}
	and
	\begin{equation*}
	\big|(\jfunk)^{(j)}(\te) \big|\lesssim\left\{
	\begin{array}{ll}
	\te^{\be+1/2-j} (k+1)^{\be+1/2},& \te\in\big(\pi-c(k+1)^{-1},\pi\big),\\
	(k+1)^{j},&\te\in\big(\frac{\pi}{3},\pi-c(k+1)^{-1}\big),
	\end{array}\right.
	\end{equation*}
	uniformly in $\te$ and $k\in\NN$.
\end{lm}
\begin{proof}
	Notice that by \eqref{eq:30} it suffices to verify the first estimate. We use the induction. For $j=0$ see \eqref{eq:27}. Assume that the claim holds for $j\in\NN$ and consider
	\begin{align*}
	\frac{d^{j+1}}{d \te^{j+1}}\jfunk(\te)&=-k_{\al,\be} \frac{d^j}{d\te^j} \phi_{k-1}^{\al+1,\be+1}(\te)+\frac{d^j}{d\te^j} \Big( \Big( \frac{2\al+1}{4}\cot\frac{\te}{2} -\frac{2\be+1}{4}\tan\frac{\te}{2}\Big) \jfunk(\te)\Big),
	\end{align*}
	where we used \eqref{eq:39}. Observe that for any given $i\in\NN$ there i
	\begin{equation*}
	\Big|\big(\tan(\te/2)\big)^{(i)}\Big|\lesssim 1\qquad \text{and}\qquad \Big|\big(\cot(\te/2)\big)^{(i)}\Big|\lesssim \te^{-(i+1)},
	\end{equation*}
	uniformly in $\te\in\big(0,\frac{2\pi}{3}\big)$. Hence, uniformly in $k\in\NN$ and $\te\in\big(0,\frac{2\pi}{3}\big)$, we have
	\begin{align*}
	\Big| \frac{d^{j+1}}{d \te^{j+1}}\jfunk(\te)\Big|\lesssim (k+1) \Big| \frac{d^{j}}{d \te^{j}}\phi_{k-1}^{\al+1,\be+1}(\te)\Big|+\sum_{i=0}^j \te^{-(i+1)} \Big| \frac{d^{j-i}}{d \te^{j-i}}\jfunk(\te)\Big| +\sum_{i=0}^j \Big| \frac{d^{i}}{d \te^{i}}\jfunk(\te)\Big|.
	\end{align*}
	Thus,
	\begin{align*}
	\Big| \frac{d^{j+1}}{d \te^{j+1}}\jfunk(\te)\Big|\lesssim  (k+1)^{\al+1/2} \te^{\al+1/2-j}\Big( (k+1)^{2} \te + \te^{-1} + 1\Big)\lesssim (k+1)^{\al+1/2} \te^{\al+1/2-j-1},
	\end{align*}
	uniformly in $k\in\NN$ and $\te\in\big(0,c(k+1)^{-1}\big)$. Similarly, 
	\begin{align*}
	\Big| \frac{d^{j+1}}{d \te^{j+1}}\jfunk(\te)\Big|\lesssim  (k+1)^{j+1},
	\end{align*}
	uniformly in $k\in\NN$ and $\te\in\big(c(k+1)^{-1},\frac{2\pi}{3}\big)$. This finishes the proof.
\end{proof}

The following result is an analogue of Lemmas \ref{lm:standard_Laguerre_sharp_estimates} and \ref{lm:Laguerre_Hermite_sharp_estimates}.

\begin{lm}\label{lm:Jacobi_fun_sharp_estimates}
	Let $j,\ell\in\NN$ and $\al,\be\geq -1/2$. There exists $c>0$ such that
	\begin{equation*}
	\frac{d^j}{d\te^j} \frac{\jfunk(\te)}{\big( \sin\frac{\te}{2}\big)^{\al+1/2-\ell}}\simeq \left\{
	\begin{array}{cc}
	(k+1)^{\al+1/2}\, \te^{\ell-j},& \ell\geq j,\\
	(-1)^{\lceil \frac{j-\ell}{2}\rceil} (k+1)^{\al + 1/2+2\lceil \frac{j-\ell}{2}\rceil} \te^{\frac{1-(-1)^{j-\ell}}{2}}, & \ell\leq j, 
	\end{array}
	\right.
	\end{equation*}
	uniformly in $k\in\NN$ and $\te\in\big(0,c(k+1)^{-1}\big)$, and
	\begin{equation*}
	\frac{d^j}{d\te^j} \frac{\jfunk(\te)}{\big( \cos\frac{\te}{2}\big)^{\be+1/2-\ell}}\simeq\left\{
	\begin{array}{cc}
	(k+1)^{\be+1/2}\, (\pi-\te)^{\ell-j},& \ell\geq j,\\
	(-1)^{\lceil \frac{j-\ell}{2}\rceil} (k+1)^{\be + 1/2+2\lceil \frac{j-\ell}{2}\rceil} (\pi-\te)^{\frac{1-(-1)^{j-\ell}}{2}}, & \ell\leq j, 
	\end{array}
	\right.
	\end{equation*}
	uniformly in $k\in\NN$ and $\te\in\big(\pi-c(k+1)^{-1},\pi\big)$.
\end{lm}
\begin{proof}
	It suffices to prove the first estimate. The reasoning is similar to the ones used in the proofs of Lemmas \ref{lm:standard_Laguerre_sharp_estimates} and \ref{lm:Laguerre_Hermite_sharp_estimates}, therefore we will only sketch it.
	
	Fix $j,\ell,\al,\be$ as in the hypothesis. We will use the induction over $j$. For $j=0$ see \cite[(A.1) and (A.2)]{Plewa_2019_JAT}. For the inductive step observe that
	\begin{align*}
	\begin{split}
	\frac{d^{j+1}}{d\te^{j+1}} \frac{\jfunk(\te)}{\big( \sin\frac{\te}{2}\big)^{\al+1/2-\ell}}&=\frac{d^j}{d\te^j}\Big( \frac{\ell}{2}\frac{\cos\frac{\te}{2} \ \jfunk(\te)}{\big( \sin\frac{\te}{2}\big)^{\al+3/2-\ell}} -\frac{2\be+1}{4}\frac{\jfunk(\te)}{\cos\frac{\te}{2} \big( \sin\frac{\te}{2}\big)^{\al-1/2-\ell}}\\
	&\qquad -\frac{k_{\al,\be} \phi_{k-1}^{\al+1,\be+1}(\te)}{\big(\sin\frac{\te}{2}\big)^{\al+1/2-\ell}}\Big).
	\end{split}
	\end{align*}
	If $j\leq \ell-1$, then the first implied component is the largest on the right hand side of the above equality. It is positive and of the desired size $(k+1)^{\al+1/2}\, \te^{\ell-j-1}$. Secondly, the case $j=\ell$ can be checked directly. On the other hand, if $j\geq\ell+1$, then for sufficiently small $c>0$ the second summand is of the sign and size 
	\begin{equation*}
	(-1)^{\lceil \frac{j-\ell-1}{2}\rceil+1} (k+1)^{\al + 1/2+2\lceil \frac{j-\ell-1}{2}\rceil} \te^{\frac{1-(-1)^{j-\ell-1}}{2}},
	\end{equation*}
	and the third
	\begin{equation*}
	(-1)^{\lceil \frac{j-\ell-1}{2}\rceil+1} (k+1)^{\al + 3/2+2\lceil \frac{j-\ell-1}{2}\rceil} \te^{\frac{1-(-1)^{j-\ell-1}}{2}}.
	\end{equation*}
	But $\lceil \frac{i-1}{2}\rceil+1=\lceil \frac{i+1}{2}\rceil$, $i\in\NN$, and hence the latter is greater and can be written as
	\begin{equation*}
	(-1)^{\lceil \frac{j+1-\ell}{2}\rceil} (k+1)^{\al + 1/2+2\lceil \frac{j+1-\ell}{2}\rceil} \te^{\frac{1-(-1)^{j+1-\ell}}{2}},
	\end{equation*}
	which finishes the inductive step.
\end{proof}

\begin{lm}\label{lm:Jacobi_fun_Linfty_derivatives}
	Let $j\in\NN$. For $\al\in\{-1/2,1/2,\ldots,j-1/2\}\cup(j-1/2,\infty)$, and $\be\geq -1/2$, there is
	\begin{equation*}
	\Big\Vert \big(\jfunk\big)^{(j)}\Big\Vert_{L^\infty(0,\frac{2\pi}{3})} \lesssim (k+1)^{j},\qquad k\in\NN,
	\end{equation*}
	whereas for $\al\geq -1/2$ and $\be\in\{-1/2,1/2,\ldots,j-1/2\}\cup (j-1/2,\infty)$, we have
	\begin{equation*}
	\Big\Vert \big(\jfunk\big)^{(j)}\Big\Vert_{L^\infty(\frac{\pi}{3},\pi)} \lesssim (k+1)^{j},\qquad k\in\NN. 
	\end{equation*}
\end{lm}
\begin{proof}
	Observe that the latter estimate follows from the former by \eqref{eq:30}. Thus we fix $j\in\NN$, $\al$, and $\be$ as in the first hypothesis. We justify the bound on $(0, \frac{2\pi}{3})$. Observe that for $\al\geq j-1/2$ it suffices to use Lemma \ref{lm:Jacobi_fun_estim}. On the other hand, if $j<\al+1/2\in\NN$, then use Lemma \ref{lm:Jacobi_fun_sharp_estimates} with $\ell=\al+1/2$. This concludes the proof.
\end{proof}

\begin{lm}\label{lm:Jacobi_fun_diff}
	Let $j\in\NN$. If $\al\in (j-1/2,j+1/2)$ and $\be\geq -1/2$, then
	\begin{equation*}
	\big|(\jfunk)^{(j)}(\te)-(\jfunk)^{(j)}(\te')\big|\lesssim (k+1)^{j+1} |\te-\te'|+(k+1)^{\al+1/2} |\te-\te'|^{\al+1/2-j},
	\end{equation*}
	uniformly in $ k\in\NN$ and $\te,\te'\in(0,\frac{2\pi}{3})$.
	Similarly, for $\al\geq -1/2$ and $\be\in (j-1/2,j+1/2)$, 
	\begin{equation*}
	\big|(\jfunk)^{(j)}(\te)-(\jfunk)^{(j)}(\te')\big|\lesssim (k+1)^{j+1} |\te-\te'|+(k+1)^{\be+1/2} |\te-\te'|^{\be+1/2-j},
	\end{equation*}
	uniformly in $ k\in\NN$ and $\te,\te'\in(\frac{\pi}{3},\pi)$.
\end{lm}
\begin{proof}
	Again, by \eqref{eq:30} we justify only the first estimate. Fix $j,\al,\be$ as in the hypothesis. For $0<\te'<\te<\frac{2\pi}{3}$ we write the difference $(\jfunk)^{(j)}(\te)-(\jfunk)^{(j)}(\te')$ as
	\begin{equation*}
	\int_{\te'}^\te \frac{d^j}{d\omega^j}\Big(  -k_{\al,\be} \phi_{k-1}^{\al+1,\be+1}(\omega)
	+\Big(\frac{2\al+1}{4}\cot\frac{\omega}{2} -\frac{2\be+1}{4}\tan\frac{\omega}{2}\Big)\jfunk(\omega) \Big)\, d\omega.
	\end{equation*}
	Thus, by the first bound in Lemma \ref{lm:Jacobi_fun_estim} we obtain
	\begin{align*}
	\big|(\jfunk)^{(j)}(\te)-(\jfunk)^{(j)}(\te')\big|\lesssim (k+1)^{j+1}|\te-\te'| + \int_{\te'}^\te \Big|\frac{d^j}{d\omega^j}\Big(\cot\frac{\omega}{2} \jfunk(\omega)\Big)\Big|\,d\omega,
	\end{align*}
	uniformly in $k\in\NN$ and $\te,\te'\in(0,\frac{2\pi}{3})$.	Using Lemma \ref{lm:Jacobi_fun_estim} we estimate the latter integral uniformly in the indicated ranges, up to a multiplicative constant, by
	\begin{equation*}
	 (k+1)^{j+1}|\te-\te'|+ (k+1)^{\al+1/2}\int_{\te'}^\te \omega^{\al-1/2-j}\,d\omega.
	\end{equation*}
	The conclusion follows since the latter integral is $\lesssim |\te-\te'|^{\al+1/2-j}$.
%	Lastly, notice that
%	\begin{equation*}
%	\int_{\te'}^\te \omega^{\al-1/2-j}\,d\omega= (\al+1/2-j)^{-1} \big|\te^{\al+1/2-j}-(\te')^{\al+1/2-j}\big|\lesssim |\te-\te'|^{\al+1/2-j},
%	\end{equation*}
%	and this finishes the proof.
\end{proof}

Now we pass to the verification of Lipschitz and $BMO$ properties of the Jacobi functions. Observe that for $\al,\be\in[-1/2,\infty)^d$, $\jfun\in L^\infty((0,\pi)^d)\subset BMO((0,\pi)^d)$. We will justify that $\jfun\in\Lambda_{\nu}\big((0,\pi)^d\big)$, $\nu>0$, for appropriate parameters $\al$ and $\be$. For this purpose we will define an extension $\tilde{\phi}^{\al,\be}_k$ of $\jfunk$ to whole $\RR$ such that $\tilde{\phi}^{\al,\be}_k\in\Lambda_{\nu}(\RR)$, and then apply Lemma \ref{lm:tensor_in_Lambda_nu} for the multi-dimensional version.

Fix $\al$ and $\be$ such that $\al+1/2,\be+1/2\in\NN$. We extend the initial definition of $\jfunk(\te)$, see \eqref{eq:jfunk_definition}, to the whole $\RR$. Note that for $j\in\NN$ and $\te\in(j\pi,(j+1)\pi)$ there holds
\begin{equation*}
\jfunk(\te)=\left\{\begin{array}{cc}
\jfunk(\te-j\pi),& j\equiv 0 \mod 4,\\
(-1)^{\be+1/2}\jfunk((j+1)\pi-\te),& j\equiv 1 \mod 4,\\
(-1)^{\al+\be+1}\jfunk(\te-j\pi),& j\equiv 2 \mod 4,\\
(-1)^{\al+1/2}\jfunk((j+1)\pi-\te),& j\equiv 3 \mod 4.
\end{array}
\right.
\end{equation*}
We remark that the second (fourth, resp.) line on the right hand side of the formula above makes sense also when $\al+1/2$ ($\be+1/2$,  resp.) is not an integer. Moreover, if $\al+1/2\in\NN$ ($\be+1/2\in\NN$, resp.), then $\jfunk(2j\pi)$ ($\jfunk((2j+1)\pi)$, resp.) is naturally defined for $j\in\NN$.
 
Now we define the extension $\tilde{\phi}^{\al,\be}_k$ of $\jfunk$. If both  $\al+1/2,\be+1/2\in\NN$, then 
\begin{equation*}
\tilde{\phi}^{\al,\be}_k(\te)=\jfunk(\te),\qquad \te\in\RR.
\end{equation*}
Secondly, if $\al+1/2\in\NN$ and $\be+1/2\notin\NN$, then
\begin{equation*}
\tilde{\phi}^{\al,\be}_k(\te)=\left\{ \begin{array}{cc}
\jfunk(\te),& \te\in(-\pi,\pi),\\
0,& \te\notin(-\pi,\pi).
\end{array} \right.
\end{equation*}
Similarly, if $\al+1/2\notin\NN$ and $\be+1/2\in\NN$, then
\begin{equation*}
\tilde{\phi}^{\al,\be}_k(\te)=\left\{ \begin{array}{cc}
\jfunk(\te),& \te\in(0,2\pi),\\
0,& \te\notin(0,2\pi).
\end{array} \right.
\end{equation*}
Finally, if both $\al+1/2,\be+1/2\notin\NN$, then we put
\begin{equation*}
\tilde{\phi}^{\al,\be}_k(\te)=\left\{ \begin{array}{cc}
\jfunk(\te),& \te\in(0,\pi),\\
0,& \te\notin(0,\pi).
\end{array} \right.
\end{equation*}
Notice that $\tilde{\phi}^{\al,\be}_k\in\mathcal{C}^{\min(\bal,\bbe)}(\RR)$, where  we used the one-off notation $\bal=\lfloor\al+1/2\rfloor$ if $\al+1/2\notin\NN$ and $\bal=\infty$ otherwise, and similarly for $\bbe$.

Now, by Lemmas \ref{lm:Jacobi_fun_Linfty_derivatives}, \ref{lm:Jacobi_fun_diff}, and \ref{lm:tensor_in_Lambda_nu}, we have the following result.
\begin{lm}\label{lm:Jacobi_fun_Lipschitz} 
	If $p\in(0,1]$ and $\al,\be\in \big(\{-1/2,1/2,\ldots,P-1/2\}\cup[d(p^{-1}-1)-1/2,\infty)\big)^d$, where $P=\lfloor d(p^{-1}-1)\rfloor$, then $\jfun\in\Lambda_{d(\frac{1}{p}-1)}((0,\pi)^d)$.
\end{lm}

\subsection{Hardy's inequality}

The one-dimensional kernels $\jRop(\te,\va)$, $r\in(0,1)$, $\te,\va\in(0,\pi)$, associated with the Jacobi functions are defined via (cf. \eqref{R_def}) 
\begin{equation*}
\jRop(\te,\va)=\sum_{k\in\NN} r^k \jfunk(\te) \jfunk(\va).
\end{equation*}
For an explicit formula see \cite[(4.1), (2.6), and (2.3)]{Plewa_2019_JAT}).

Notice that by Parseval's identity and interchanging the differentiation with the summation, which is allowed due to Lemma \ref{lm:Jacobi_fun_Linfty_derivatives} and the Lebesgue dominated convergence theorem, and we obtain the following lemma. 

\begin{lm}\label{lm:Jacobi_fun_Rop}
	If $j\in\NN$, and $\al,\be\in\{-1/2,1/2,\ldots,j-1/2\}\cup (j-1/2,\infty)$, then
\begin{equation*}
\sup_{\te\in(0,\pi)} \big\Vert \partial_\te^j \jRop(\te,\cdot)\big\Vert_{L^2((0,\pi))}\lesssim (1-r)^{-(j+1)/2},\qquad r\in(0,1).
\end{equation*}
\end{lm}

%\begin{equation}\label{eq:Jacobi_fun_R_definition}
%\jRop(\te,\va)=r^{-(\al+\be+1)/2} \Big(\sin\frac{\te}{2}\sin\frac{\va}{2}\Big)^{\al+1/2}\Big(\cos\frac{\te}{2}\cos\frac{\va}{2}\Big)^{\be+1/2} \mathcal{H}^{\al,\be}_{-\log r}(\te,\va).
%\end{equation}
%Here,
%\begin{equation}\label{Pois-Jac_nonsym_pol}
%\tpois(\te,\va)=\frac{\sinh \frac{t}{2}}{4^\eta\mu_{\al,\be}(0,\pi)}\iint\limits_{[-1,1]^2} \frac{d\Pi_\al(u)\,d\Pi_\be(v)}{(\cosh \frac{t}{2}-1+q(\te,\va,u,v))^{\al+\be+2}},
%\end{equation}
%where
%\begin{equation*}
%q(\te,\va,u,v)=1-u\sin\frac{\te}{2}\sin\frac{\va}{2}-v\cos\frac{\te}{2}\cos\frac{\va}{2},\qquad \te,\va\in(0,\pi),\ u,v\in[-1,1].
%\end{equation*}
%and the measures $\Pi_\al$ and $\Pi_\beta$ are define similarly as in \eqref{eq:definition_Pi_nu} (for the details see \cite[pp.~189-190]{Nowak&Sjogren&Szarek_2015_CA}).

%Now we will estimate the derivatives of $\jRop$. In order to do this we will make use of the following theorem.
%
%\begin{thm}[{\cite[Lemmma~3.4]{Castro&Nowak&Szarek_2016_JFAA}}]\label{thm:Jacobi_pol_poissonkernel_estim}
%	If $\al,\be\geq -1/2$ and $j\in\NN$, then
%	\begin{equation*}
%	\big|\partial_\te^j\tpois(\te,\va)\big|\simeq\big(t^2+\te^2+\va^2\big)^{-\al-1/2} \big(t^2+(\pi-\te)^2+(\pi-\va)^2\big)^{-\be-1/2}\frac{t}{(t^2+(\te-\va)^2)^{1+j/2}},
%	\end{equation*}
%	uniformly in $\te,\va\in(0,\pi)$ and $0<t\leq 1$.
%\end{thm}
%
%In fact this theorem is valid for $\al,\be>-1$, but we shall focus only on the restricted range, since if $\al$ or $\be$ is in $(-1,-1/2)$, then $\jfunk$ is not bounded.  

In order to verify appropriate version of \eqref{cond:C} we firstly estimate differences of the derivatives of $\jRop(\te,\va)$. We remark that in order to prove the below-stated proposition, one could use \cite[Lemma~3.4]{Castro&Nowak&Szarek_2016_JFAA} and the explicit form of the investigated kernels. However, Lemma \ref{lm:Jacobi_fun_diff} gives this result much quicker.

\begin{prop}
If $j\in\NN$ and $\al,\be\in\{-1/2,1/2,\ldots,j-1/2\}\cup (j-1/2,\infty)$, then
\begin{align*}
&\big\Vert \partial_\te^j \jRop(\te,\cdot)-\partial_\te^j \jRop(\te',\cdot)\big\Vert_{L^2((0,\pi))}\\
&\quad\lesssim (1-r)^{-(j+3/2)}|\te-\te'|+(1-r)^{-(\al+1)}|\te-\te'|^{\al+1/2-j}+(1-r)^{-(\be+1)}|\te-\te'|^{\be+1/2-j},
\end{align*}
uniformly in $r\in(0,1)$ and $\te,\te'\in(0,\pi)$, where the second (third, resp.) summand on the right hand side of the estimate appears only if $\al$ ($\be$, resp.) belongs to $(j-1/2,j+1/2)$. 
\end{prop}
\begin{proof}
	In the case $\al,\be\notin (j-1/2,j+1/2)$ we simply use the mean value theorem and Lemma \ref{lm:Jacobi_fun_Rop}. On the other hand, if one or both of the parameters $\al$ and $\beta$ is in $(j-1/2,j+1/2)$, then we apply Parseval's identity and Lemma \ref{lm:Jacobi_fun_diff}.
\end{proof}

Now the following proposition follows easily (compare with Propositions \ref{prop:standard_Laguerre_cond_C} and \ref{prop:Laguerre_Hermite_cond_C}).

\begin{prop}\label{prop:Jacobi_fun_cond_C}
	If $k\in\NN$ and $\al,\be\in\big(\{-1/2,1/2,\ldots,k-1/2\}\cup (k-1/2,\infty)\big)^d$, then
	\begin{align*}
	\Big\Vert &\jRop(\te,\cdot)-\sum_{\stackrel{i_1,\ldots,i_d\geq 0}{i_1+\ldots+ i_d\leq k}} \frac{(i_1+\ldots+i_d)!}{i_1!\cdot\ldots\cdot i_d!}\partial^{i_1}_{\te_1}\ldots \partial^{i_d}_{\te_d} \jRop(\te',\cdot)\prod_{j=1}^d (\te_j-\te_j')^{i_j}\Big\Vert_{L^2((0,\pi)^d)}\\
	& \lesssim \sum_{\delta\in\Delta^{\al,\be}_k}(1-r)^{-\frac{d+2k+2\delta}{2}}|\te-\te'|^{k+\delta},
	\end{align*}		
	uniformly in $r\in(0,1)$ and $\te,\te'\in (0,\pi)^d$, where
	\begin{align*}
	\Delta^\al_k=\{1\}&\cup\{\al_i+1/2-k:\ \al_i\in(k-1/2,k+1/2) \}\\
	&\cup\{\be_i+1/2-k:\ \be_i\in(k-1/2,k+1/2) \}.
	\end{align*}
\end{prop}

Now we are ready to state Hardy's inequality associated with the Jacobi trigonometric functions.

\begin{thm}\label{thm:main_Jacobi_fun}
	Let $p\in(0,1)$, $s\in[p,2]$, and $P=\lfloor d(p^{-1}-1)\rfloor$. For
	\begin{equation*}
	\al,\be\in\big(\{-1/2,1/2,\ldots,P-1/2\}\cup (d(p^{-1}-1)-1/2,\infty)\big)^d,
	\end{equation*}
	there holds
	\begin{equation*}
	\sum_{n\in\NN^d}\frac{|\langle f,\jfun\rangle|^s}{(\ven+1)^E}\lesssim \Vert f\Vert^s_{H^p((0,\pi)^d)}, \qquad f\in H^p((0,\pi)^d),
	\end{equation*}
	where $E= d+sd\big(p^{-1}-1\big)$, and the exponent is sharp.
\end{thm}
\begin{proof}
	Similarly as in the proofs of Theorems \ref{thm:main_Laguerre_standard} and \ref{thm:main_Laguerre_Hermite} the inequality follows from Theorem \ref{thm:Hp_general} and Proposition \ref{prop:Jacobi_fun_cond_C}, whereas sharpness is a consequence of Propositions \ref{prop:general_sharpness}, \ref{prop:general_sharpness_2} and Lemma \ref{lm:Jacobi_fun_sharp_estimates}. As in the previous sections, we exclude the cases $\al_i+1/2,\be_i+1/2=d(p^{-1}-1)$, unless $d(p^{-1}-1)$ is an integer, see Remark \ref{rem:sharpness}. 
\end{proof}

\bibliographystyle{acm}
\bibliography{Plewa_Hardy_ineq_on_Hp}

\begin{thebibliography}{10}

\bibitem{Askey_Wainger_1965_AJM}
{\sc Askey, R., and Wainger, S.}
\newblock Mean convergence of expansions in {L}aguerre and {H}ermite series.
\newblock {\em Amer. J. Math. 87\/} (1965), 695--708.

\bibitem{Balasubramanian_Radha_2005_JIPAA}
{\sc Balasubramanian, R., and Radha, R.}
\newblock Hardy-type inequalities for {H}ermite expansions.
\newblock {\em J. Inequal. Pure Appl. Math. 6}, 1 (2005), 1--4.

\bibitem{Campanato_1964_ASNSP(H^p-duality)}
{\sc Campanato, S.}
\newblock Propriet\`a di una famiglia di spazi funzionali.
\newblock {\em Ann. Scuola Norm. Sup. Pisa Cl. Sci. (3) 18\/} (1964), 137--160.

\bibitem{Castro&Nowak&Szarek_2016_JFAA}
{\sc Castro, A.~J., Nowak, A., and Szarek, T.~Z.}
\newblock Riesz-{J}acobi transforms as principal value integrals.
\newblock {\em J. Fourier Anal. Appl. 22}, 3 (2016), 493--541.

\bibitem{Chang&Dafni&Stein_1999_TAMS}
{\sc Chang, D.-C., Dafni, G., and Stein, E.~M.}
\newblock Hardy spaces, {BMO}, and boundary value problems for the {L}aplacian
  on a smooth domain in {$\mathbf{R}^n$}.
\newblock {\em Trans. Amer. Math. Soc. 351}, 4 (1999), 1605--1661.

\bibitem{Chang&Krantz&Stein_1993_JFA}
{\sc Chang, D.-C., Krantz, S.~G., and Stein, E.~M.}
\newblock {$H^p$} theory on a smooth domain in {${\bf R}^N$} and elliptic
  boundary value problems.
\newblock {\em J. Funct. Anal. 114}, 2 (1993), 286--347.

\bibitem{Ciaurri&Nowak&Stempak_2007_MZ}
{\sc Ciaurri, O., Nowak, A., and Stempak, K.}
\newblock Jacobi transplantation revisited.
\newblock {\em Math. Z. 257}, 2 (2007), 355--380.

\bibitem{Ciaurri&Roncal&Thangavelu_2018_PEMS}
{\sc Ciaurri, O., Roncal, L., and Thangavelu, S.}
\newblock Hardy-type inequalities for fractional powers of the
  {D}unkl-{H}ermite operator.
\newblock {\em Proc. Edinb. Math. Soc. (2) 61}, 2 (2018), 513--544.

\bibitem{Coifman_Weiss_1977_BAMS}
{\sc Coifman, R.~R., and Weiss, G.}
\newblock Extensions of {H}ardy spaces and their use in analysis.
\newblock {\em Bull. Amer. Math. Soc. 83}, 4 (1977), 569--645.

\bibitem{Fefferman_1971_BAMS_duality}
{\sc Fefferman, C.}
\newblock Characterizations of bounded mean oscillation.
\newblock {\em Bull. Amer. Math. Soc. 77\/} (1971), 587--588.

\bibitem{Fefferman&Stein_1972_AM}
{\sc Fefferman, C., and Stein, E.~M.}
\newblock {$H\sp{p}$} spaces of several variables.
\newblock {\em Acta Math. 129}, 3-4 (1972), 137--193.

\bibitem{Garcia-Cuerva&RdFrancia_1985_N-HPA}
{\sc Garc\'{\i}a-Cuerva, J., and Rubio~de Francia, J.~L.}
\newblock {\em Weighted norm inequalities and related topics}, vol.~116 of {\em
  North-Holland Mathematics Studies}.
\newblock North-Holland Publishing Co., Amsterdam, 1985.

\bibitem{Grafakos_Modern}
{\sc Grafakos, L.}
\newblock {\em Modern {F}ourier analysis}, third~ed., vol.~250 of {\em GTM}.
\newblock Springer, New York, 2014.

\bibitem{Hardy&Littlewood_1927_MA}
{\sc Hardy, G.~H., and Littlewood, J.~E.}
\newblock Some new properties of {F}ourier constants.
\newblock {\em Math. Ann. 97}, 1 (1927), 159--209.

\bibitem{Kanjin_1997_BLMS}
{\sc Kanjin, Y.}
\newblock Hardy's inequalities for {H}ermite and {L}aguerre expansions.
\newblock {\em Bull. London Math. Soc. 29}, 3 (1997), 331--337.

\bibitem{Kanjin_2011_JMSJ}
{\sc Kanjin, Y.}
\newblock Hardy's inequalities for {H}ermite and {L}aguerre expansions
  revisited.
\newblock {\em J. Math. Soc. Japan 63}, 3 (2011), 753--767.

\bibitem{Kanjin_Sato_2001_BLMS(Jacobi_Paley)}
{\sc Kanjin, Y., and Sato, K.}
\newblock Paley's inequality for the {J}acobi expansions.
\newblock {\em Bull. London Math. Soc. 33}, 4 (2001), 483--491.

\bibitem{Kanjin_Sato_2004_MIA(Jacobi_Hardy)}
{\sc Kanjin, Y., and Sato, K.}
\newblock Hardy's inequality for {J}acobi expansions.
\newblock {\em Math. Inequal. Appl. 7}, 4 (2004), 551--555.

\bibitem{Lebedev1972}
{\sc Lebedev, N.~N.}
\newblock {\em Special functions and their applications}.
\newblock Dover Publications, Inc., New York, 1972.

\bibitem{LiShi_2014_CA}
{\sc Li, Z., and Shi, Y.}
\newblock Multipliers of {H}ardy spaces associated with generalized {H}ermite
  expansions.
\newblock {\em Constr. Approx. 39}, 3 (2014), 517--540.

\bibitem{LiYuShi_2015_JFAA}
{\sc Li, Z., Yu, Y., and Shi, Y.}
\newblock The {H}ardy inequality for {H}ermite expansions.
\newblock {\em J. Fourier Anal. Appl. 21}, 2 (2015), 267--280.

\bibitem{Lu_1995_FourLectures}
{\sc Lu, S.~Z.}
\newblock {\em Four lectures on real {$H^p$} spaces}.
\newblock World Scientific Publishing Co., Inc., River Edge, NJ, 1995.

\bibitem{Miyachi_1990_SM}
{\sc Miyachi, A.}
\newblock {$H^p$} spaces over open subsets of {${\bf R}^n$}.
\newblock {\em Studia Math. 95}, 3 (1990), 205--228.

\bibitem{Muckenhoupt_1970}
{\sc Muckenhoupt, B.}
\newblock Mean convergence of {H}ermite and {L}aguerre series. {I}{I}.
\newblock {\em Trans. Amer. Math. Soc. 147}, 2 (1970), 433--460.

\bibitem{Muckenhoupt_1986_MAMS}
{\sc Muckenhoupt, B.}
\newblock Transplantation theorems and multiplier theorems for {J}acobi series.
\newblock {\em Mem. Amer. Math. Soc. 64}, 356 (1986), iv+86.

\bibitem{Nowak&Stempak_2007_JFA}
{\sc Nowak, A., and Stempak, K.}
\newblock Riesz transforms and conjugacy for {L}aguerre function expansions of
  {H}ermite type.
\newblock {\em J. Funct. Anal. 244}, 2 (2007), 399--443.

\bibitem{Nowak_Szarek_2012_JMAA}
{\sc Nowak, A., and Szarek, T.~Z.}
\newblock Calder\'{o}n-{Z}ygmund operators related to {L}aguerre function
  expansions of convolution type.
\newblock {\em J. Math. Anal. Appl. 388}, 2 (2012), 801--816.

\bibitem{Plewa_2018_sharpLaguerre_arxiv}
{\sc {Plewa}, P.}
\newblock Sharp {H}ardy's type inequality for {L}aguerre expansion.
\newblock arXiv:1810.08138, 2018.

\bibitem{Plewa_2019_JFAA}
{\sc Plewa, P.}
\newblock Hardy's inequality for {L}aguerre expansions of {H}ermite type.
\newblock {\em J. Fourier Anal. Appl. 25}, 4 (2019), 1855--1873.

\bibitem{Plewa_2020_TJM}
{\sc {Plewa}, P.}
\newblock On {H}ardy's inequality for {H}ermite expansions.
\newblock {\em Taiwanese J. Math. 24}, 2 (2020), 301--315.

\bibitem{Plewa_2019_JAT}
{\sc Plewa, P.}
\newblock Sharp hardy’s inequality for jacobi and symmetrized jacobi
  trigonometric expansions.
\newblock {\em J. of Approx. Theory 256\/} (2020), 105422.
\newblock https://doi.org/10.1016/j.jat.2020.105422.

\bibitem{Radha_2000_TJM}
{\sc Radha, R.}
\newblock Hardy-type inequalities.
\newblock {\em Taiwanese J. Math. 4}, 3 (2000), 447--456.

\bibitem{Radha_Thangavelu_2004_PAMS}
{\sc Radha, R., and Thangavelu, S.}
\newblock Hardy's inequalities for {H}ermite and {L}aguerre expansions.
\newblock {\em Proc. Amer. Math. Soc. 132}, 12 (2004), 3525--3536.

\bibitem{Rudin_1991_FunctionalAnalysis}
{\sc Rudin, W.}
\newblock {\em Functional analysis}, second~ed.
\newblock McGraw-Hill, New York, 1991.

\bibitem{Satake_2000_JMSJ}
{\sc Satake, M.}
\newblock Hardy's inequalities for {L}aguerre expansions.
\newblock {\em J. Math. Soc. Japan 52}, 1 (2000), 17--24.

\bibitem{Shi&Li_2016_JMSJ_2}
{\sc Shi, Y., and Li, Z.}
\newblock Coefficient multipliers of {$H^1$} into {$\ell^q$} associated with
  {L}aguerre expansions.
\newblock {\em J. Math. Soc. Japan 68}, 2 (2016), 797--805.

\bibitem{Shi&Li_2016_JMSJ}
{\sc Shi, Y., and Li, Z.}
\newblock Multipliers of {H}ardy spaces associated with {L}aguerre expansions.
\newblock {\em J. Math. Soc. Japan 68}, 1 (2016), 91--99.

\bibitem{Stein_HarmonicAnalysis}
{\sc Stein, E.~M.}
\newblock {\em Harmonic analysis: real-variable methods, orthogonality, and
  oscillatory integrals}.
\newblock Princeton Univ. Press, Princeton, NJ, 1993.

\bibitem{Stempak_1994_Tohoku}
{\sc Stempak, K.}
\newblock Heat-diffusion and {P}oisson integrals for {L}aguerre expansions.
\newblock {\em Tohoku Math. J. (2) 46}, 1 (1994), 83--104.

\bibitem{Szego1959}
{\sc Szeg\"{o}, G.}
\newblock {\em Orthogonal polynomials}.
\newblock AMS Colloquium Publications, Vol. 23. Revised ed. AMS, Providence,
  R.I., 1959.

\bibitem{Uchiyama_Springer_2001}
{\sc Uchiyama, A.}
\newblock {\em Hardy spaces on the {E}uclidean space}.
\newblock Springer-Verlag, Tokyo, 2001.

\end{thebibliography}

\end{document}